\documentclass[a4paper,12pt]{article}
\title{\bf{Curve counting theories via stable objects~I. 
DT/PT correspondence}
}
\date{}
\author{Yukinobu Toda}

\usepackage{makeidx}

\usepackage{latexsym}
\usepackage{amscd}
\usepackage{amsmath}
\usepackage{amssymb}
\usepackage{amsthm}
\usepackage{float}
\usepackage[all]{xy}
\usepackage{graphicx}

\usepackage{array}
\usepackage{amscd}
\usepackage[all]{xy}
\usepackage{makeidx}
\usepackage{latexsym}
\DeclareFontFamily{U}{rsfs}{%
\skewchar\font127}
\DeclareFontShape{U}{rsfs}{m}{n}{%
<-6>rsfs5<6-8.5>rsfs7<8.5->rsfs10}{}
\DeclareSymbolFont{rsfs}{U}{rsfs}{m}{n}
\DeclareSymbolFontAlphabet
{\mathrsfs}{rsfs}
\DeclareRobustCommand*\rsfs{%
\@fontswitch\relax\mathrsfs}
\theoremstyle{plain}
\newtheorem{thm}{Theorem}[section]
\newtheorem{prop}[thm]{Proposition}
\newtheorem{lem}[thm]{Lemma}

\newtheorem{defi}[thm]{Definition}
\newtheorem{rmk}[thm]{Remark}
\newtheorem{cor}[thm]{Corollary}

\newtheorem{step}{Step}
\newtheorem{sstep}{Step}
\newtheorem{ssstep}{Step}

\newtheorem{prop-defi}[thm]{Proposition-Definition}
\newtheorem{thm-defi}[thm]{Theorem-Definition}
\newtheorem{lem-defi}[thm]{Lemma-Definition}

\newtheorem{assum}[thm]{Assumption}
\newtheorem{conj}[thm]{Conjecture}
\newtheorem{exam}[thm]{Example}

\newdimen\argwidth
\def\db[#1\db]{
 \setbox0=\hbox{$#1$}\argwidth=\wd0
 \setbox0=\hbox{$\left[\box0\right]$}
  \advance\argwidth by -\wd0
 \left[\kern.3\argwidth\box0 \kern.3\argwidth\right]}

\newcommand{\aA}{\mathcal{A}}
\newcommand{\bB}{\mathcal{B}}
\newcommand{\cC}{\mathcal{C}}
\newcommand{\dD}{\mathcal{D}}
\newcommand{\eE}{\mathcal{E}}
\newcommand{\fF}{\mathcal{F}}

\newcommand{\hH}{\mathcal{H}}

\newcommand{\mM}{\mathcal{M}}

\newcommand{\oO}{\mathcal{O}}
\newcommand{\pP}{\mathcal{P}}
\newcommand{\qQ}{\mathcal{Q}}

\newcommand{\tT}{\mathcal{T}}
\newcommand{\uU}{\mathcal{U}}
\newcommand{\vV}{\mathcal{V}}
\newcommand{\wW}{\mathcal{W}}
\newcommand{\xX}{\mathcal{X}}
\newcommand{\yY}{\mathcal{Y}}
\newcommand{\zZ}{\mathcal{Z}}
\newcommand{\lr}{\longrightarrow}

\newcommand{\Supp}{\mathop{\rm Supp}\nolimits}
\newcommand{\Hom}{\mathop{\rm Hom}\nolimits}

\newcommand{\dL}{\mathbf{L}}

\newcommand{\Pic}{\mathop{\rm Pic}\nolimits}

\newcommand{\ch}{\mathop{\rm ch}\nolimits}
\newcommand{\rk}{\mathop{\rm rk}\nolimits}

\newcommand{\Ext}{\mathop{\rm Ext}\nolimits}

\newcommand{\Coh}{\mathop{\rm Coh}\nolimits}

\newcommand{\cneq}{\mathrel{\raise.095ex\hbox{:}\mkern-4.2mu=}}
\newcommand{\eqcn}{\mathrel{=\mkern-4.5mu\raise.095ex\hbox{:}}}

\newcommand{\NE}{\mathop{\rm NE}\nolimits}

\newcommand{\Cok}{\mathop{\rm cok}\nolimits}

\newcommand{\Aut}{\mathop{\rm Aut}\nolimits}

\newcommand{\Stab}{\mathop{\rm Stab}\nolimits}

\newcommand{\kakkoS}{\mathbb{C}\db[S\db]}
\newcommand{\kakkoSl}{\mathbb{C}\db[S_{\lambda}\db]}
\newcommand{\kakkoSll}{\mathbb{C}\db[S_{\lambda'}\db]}
\newcommand{\kakkoT}{\mathbb{C}\db[T\db]}

\newcommand{\DT}{\mathop{\rm DT}\nolimits}
\newcommand{\PT}{\mathop{\rm PT}\nolimits}

\newcommand{\modu}{\mathop{\rm mod}\nolimits}
\newcommand{\End}{\mathop{\rm End}\nolimits}

\newcommand{\Slice}{\mathop{\rm Slice}\nolimits}

\newcommand{\Imm}{\mathop{\rm Im}\nolimits}
\newcommand{\imm}{\mathop{\rm im}\nolimits}

\newcommand{\Ker}{\mathop{\rm ker}\nolimits}

\newcommand{\GL}{\mathop{\rm GL}\nolimits}

\newcommand{\tr}{\mathop{\rm tr}\nolimits}
\newcommand{\ex}{\mathop{\rm ex}\nolimits}

\newcommand{\cl}{\mathop{\rm cl}\nolimits}
\begin{document}
\maketitle
\begin{abstract}
The Donaldson-Thomas invariant is a
curve counting invariant on Calabi-Yau 3-folds
via ideal sheaves. 
 Another counting invariant via stable pairs 
is introduced by Pandharipande and Thomas, which 
counts pairs of curves and divisors on them. 
These two theories are conjecturally equivalent
via generating functions, called DT/PT 
correspondence. 
In this paper, we show the Euler characteristic 
version of DT/PT correspondence, using the 
notion of weak stability conditions and the
wall-crossing formula. 
\end{abstract}
\section{Introduction}
The purpose of this paper 
is to study 
curve counting on Calabi-Yau 3-folds via 
wall-crossing phenomena in the derived category. 
We will study the generating series of 
Donaldson-Thomas type invariants without 
virtual fundamental cycles, i.e. 
the Euler characteristics of the relevant moduli spaces. 
The main result is to show the 
Euler characteristic version of 
Pandharipande-Thomas conjecture~\cite[Conjecture~3.3]{PT}, 
which claims the equality of the generating series of 
Donaldson-Thomas invariants and 
counting invariants of stable pairs. 
In a subsequent paper~\cite{Tcurve2}, we will apply the method 
used in this paper to show the transformation formula 
of our generating series under flops and the 
generalized McKay correspondence
by Van den Bergh~\cite{MVB}. 
 
\subsection{Donaldson-Thomas invariant}
Let $X$ be a smooth projective Calabi-Yau 3-fold over 
$\mathbb{C}$, i.e. the canonical line bundle $\bigwedge^3 T_X^{\ast}$ 
is trivial. For a homology class $\beta \in H_2(X, \mathbb{Z})$ 
and $n\in \mathbb{Z}$, the moduli space which 
defines DT-invariant is the classical Hilbert scheme, 
$$I_{n}(X, \beta)=\left\{\begin{array}{ll} \mbox{ subschemes }
C\subset X, \dim C \le 1 \\
\mbox{ with } [C]=\beta, \ \chi(\oO_C)=n. \end{array}
\right\}.
$$
In other words, $I_n(X, \beta)$ is the moduli space of 
rank one torsion free sheaves $I \in \Coh(X)$ which 
satisfies $\det I=\oO_X$ and 
$$\ch(I)=(1, 0, -\beta, -n) \in H^{0}\oplus H^2 \oplus H^4 \oplus H^6.$$
Here we have regarded $\beta$ as an element of $H^4(X, \mathbb{Z})$
via the Poincar\'e duality. 
In fact such a sheaf $I$ is isomorphic to 
the ideal sheaf $I_C\subset \oO_X$
for a subscheme $C\subset X$ with $\dim C\le 1$, 
$[C]=\beta$ and $\chi(\oO_C)=n$. 
The moduli space $I_n(X, \beta)$ is projective and has a 
symmetric obstruction theory~\cite{Thom}. The associated 
virtual 
fundamental cycle has virtual dimension zero, and the
integration along it defines
the DT-invariant, 
$$I_{n, \beta}=\int_{[I_n(X, \beta)^{\rm{vir}}]}1 \in \mathbb{Z}.$$
We consider the generating series, 
$$\DT(X)=\sum_{n, \beta}I_{n, \beta}x^n y^{\beta}.$$
Let $\DT_{0}(X)$ be the contributions from
0-dimensional subschemes, 
$$\DT_{0}(X)=\sum_{n}I_{n, 0}x^n.$$
This is computed in~\cite{BBr}, \cite{Li}, \cite{LP},  
$$\DT_{0}(X)=M(-x)^{\chi(X)}, \quad M(x)=\prod_{k\ge 1}\frac{1}{(1-x^k)^k}.$$
The reduced Donaldson-Thomas theory is defined by, 
$$\DT'(X)\cneq \frac{\DT(X)}{\DT_{0}(X)}=
\sum_{\beta}\DT_{\beta}'(X)y^{\beta},$$
where $\DT_{\beta}'(X)$ is a Laurent series of $x$. 
The MNOP conjecture~\cite{MNOP} 
states that $\DT_{\beta}'(X)$ is the Laurent expansion of a 
rational function of $x$ invariant under $x\leftrightarrow 1/x$, 
and $\DT'(X)$ coincides with 
the generating series of Gromov-Witten invariants
after a suitable change of variables. 
\subsection{Pandharipande-Thomas theory}
Another curve counting theory via stable pairs is introduced by 
Pandharipande and Thomas~\cite{PT} in order to 
give a geometric understanding of the reduced DT-theory. 
By definition, a stable pair
$(F, s)$ consists of pure 1-dimensional sheaf
$F$ and a morphism $s\colon \oO_X \to F$ 
with 0-dimensional cokernel. 
In~\cite{PT}, the moduli space 
$$P_{n}(X, \beta)=\left\{ \begin{array}{l}
\mbox{ stable pairs }(F, s) \mbox{ with } \\
\mbox{ }[F]=\beta, \ \chi(F)=n. 
\end{array} \right\},$$
is shown to be a projective variety, and 
has a symmetric obstruction theory by viewing 
stable pairs as two-term complexes, 
\begin{align}\label{view}
(\cdots \to 0 \to \oO_X \stackrel{s}{\to} F \to 0 \cdots )
\in D^b(\Coh(X)).
\end{align}
Integrating along the virtual fundamental cycle defines
the invariant, 
$$P_{n, \beta}=\int_{[P_n(X, \beta)^{\rm{vir}}]} \in \mathbb{Z}.$$
We consider the generating series, 
\begin{align*}
\PT(X)&\cneq \sum_{n, \beta}P_{n, \beta}x^n y^{\beta} 
=\sum_{\beta}\PT_{\beta}(X)y^{\beta}, 
\end{align*}
where $\PT_{\beta}(X)$ is a Laurent series 
of $x$. In~\cite[Conjecture~3.3]{PT}, 
Pandharipande and Thomas state the following conjecture. 
\begin{conj}\emph{(\cite[Conjecture~3.3]{PT})}\label{conj:DTPT}
We have the equality of the generating series, 
$$\DT'(X)=\PT(X).$$
\end{conj}

\subsection{Main theorem}
In this paper, 
we study the series, 
$$\widehat{\DT}(X)=\sum_{n, \beta}\chi(I_n(X, \beta))x^n y^{\beta},$$
where $\chi(\ast)$ is the topological Euler characteristic. 
We can similarly define the series $\widehat{\DT}_{0}(X)$, 
$\widehat{\DT}'(X)$, 
$\widehat{\PT}(X)$, which are Euler characteristic
versions of $\DT_0(X)$, $\DT'(X)$, 
$\PT(X)$ respectively.  
The series $\widehat{\DT}(X)$ is 
closely related to $\DT(X)$ in the following sense. 
\begin{itemize}
\item If $I_n(X, \beta)$ is non-singular and connected, 
we have 
$$I_{n, \beta}=(-1)^{\dim I_n(X, \beta)}\chi(I_n(X, \beta)).$$
\item In general, there is Behrend's constructible function~\cite{Beh}, 
$$\nu \colon I_{n}(X, \beta) \to \mathbb{Z}, $$
such that $I_{n, \beta}$ is written as 
$$I_{n, \beta}=\sum_{n\in \mathbb{Z}}
n\chi(\nu^{-1}(n)).$$
\item As for $\widehat{\DT}_{0}(X)$, 
we have $\widehat{\DT}_{0}(X)=M(x)^{\chi(X)}$, 
so it is obtained from $\DT_{0}(X)$ by $x\leftrightarrow -x$. 
(cf.~\cite{Cheah}.)
\end{itemize}
Our main theorem is the following. 
\begin{thm}\label{thm:main}\emph{\bf{[Theorem~\ref{main:dtpt}]}}
We have the equality of the generating series, 
$$\widehat{\DT}'(X)=\widehat{\PT}(X).$$
\end{thm}
In~\cite[Corollary~1.4]{Tolim2}, 
the author showed the rationality of the 
series $\widehat{\PT}_{\beta}(X)$. 
Hence 
we obtain the following.
\begin{cor}
The series $\widehat{\DT}'_{\beta}(X)$ is the Laurent 
expansion of a rational function of $x$, invariant 
under $x\leftrightarrow 1/x$. 
\end{cor}
Note that the above result is conjectured in~\cite[Conjecture~1.1]{Li-Qin}. 
\subsection{Idea of the proof of
Theorem~\ref{thm:main}}
Our proof is based on the idea of 
Pandharipande and Thomas~\cite[Section~3]{PT}
to use Joyce's wall-crossing formula~\cite{Joy4} in the 
space of Bridgeland's stability conditions~\cite{Brs1} on 
the triangulated category $D^b(\Coh(X))$.
Suppose that there is a
stability condition $\sigma$ on $D^b(\Coh(X))$ such that
the ideal sheaf
$I_C$ for 
a 1-dimensional subscheme $C\subset X$
is $\sigma$-stable. 
If there is a 0-dimensional subsheaf $Q\subset \oO_C$, 
i.e. $\oO_X \to \oO_C$ is not a stable pair, then
there is a distinguished triangle 
in $D^b(\Coh(X))$, 
\begin{align}\label{QI}
 Q[-1] \lr I_C \lr I_{C'}, 
 \end{align}
where $\oO_{C'}=\oO_C/Q$. Then Pandharipande and Thomas claim that 
we can deform stability conditions from $\sigma$
to another stability condition $\tau$, 
such that 
the sequence (\ref{QI}) destabilizes $I_C$
with respect to $\tau$.  
Instead if we consider the flipped sequence
$$I_{C'} \lr E \lr Q[-1], $$
then the object $E$ should become $\tau$-stable. 
The object $E$ is isomorphic to 
a two-term complex (\ref{view}) determined by 
a stable pair, so $\sigma$ corresponds to the DT-theory and 
$\tau$ corresponds to the PT-theory. 
In this way, we can see that
the relationship between counting invariants of 
$\sigma$-stable objects and $\tau$-stable objects 
is relevant to Conjecture~\ref{conj:DTPT}.
In principle, there should exist a wall and chamber structure on 
the space of stability conditions, so that 
the counting invariants are constant on chambers but 
jump at walls. 
The transformation formula of counting invariants
under change of stability conditions, called the
\textit{wall-crossing formula}, 
is studied by Joyce~\cite{Joy4} in the case of 
abelian categories. 
As pointed out in~\cite[Section~3]{PT}, there are two  
issues in applying Joyce's theory.  
\begin{itemize}
\item 
We need to extend Joyce's work to stability conditions on 
the triangulated category
$D^b(\Coh(X))$. However there are no known 
examples of Bridgeland's stability conditions 
on $D^b(\Coh(X))$ for a projective Calabi-Yau 3-fold $X$. 
\item Joyce studies wall-crossing formula of 
counting invariants without 
virtual fundamental cycles.  
We need to establish a similar formula for invariants 
involving virtual classes, or Behrend's constructible functions. 
\end{itemize}
In this paper, we deal with the first issue. 
The idea consists of two parts.  
\begin{itemize}
\item Instead of working with $D^b(\Coh(X))$, we study the 
triangulated subcategory, 
$$\dD_X=\langle \oO_X, \Coh_{\le 1}(X) \rangle_{\tr} \subset D^b(\Coh(X)),$$
i.e. the smallest triangulated subcategory which contains 
$\oO_X$ and $F\in \Coh(X)$ with $\dim \Supp(F)\le 1$. 
Note that ideal sheaves $I_C$ and two-term complexes (\ref{view}) 
are contained in $\dD_X$. 
On the triangulated category $\dD_X$, we are
able to construct
Bridgeland's stability conditions.  
\item Although there are stability conditions on $\dD_X$, 
still there are technical
difficulties to study 
stability conditions on $\dD_X$
and wall-crossing phenomena. 
So we introduce the space of weak stability conditions on
triangulated categories, which generalizes Bridgeland's
stability conditions. It is easier to construct weak 
stability conditions than usual stability conditions, and 
the wall-crossing formula also becomes much more amenable. 
\end{itemize}
Based on these two ideas, we can justify the discussion
of Pandharipande and Thomas~\cite[Section~3]{PT} and give the 
proof of Theorem~\ref{thm:main}.
For the application in a subsequent paper~\cite{Tcurve2}, we 
show the wall-crossing formula in the space of weak stability 
conditions on $\dD_X$ under a general setting. 

As for the second issue, there are important progress recently. 
In~\cite{K-S}, Kontsevich and Soibelman establish the wall-crossing 
formula for motivic Donaldson-Thomas invariants, which 
essentially involves virtual classes~\cite[Theorem~7]{K-S}. 
Although their main result~\cite[Theorem~7]{K-S}
 relies
on the unsolved conjecture on Motivic Milnor 
fibers~\cite[Conjecture~4]{K-S}, 
their work is applied for numerical
Donaldson-Thomas invariants once we know the $l$-adic 
version of~\cite[Conjecture~4]{K-S}, which 
is solved in~\cite[Proposition~9]{K-S}. 
(However we still need an orientation data~\cite[Section~5]{K-S}
for the application of the result of Kontsevich and Soibelman.) 
If we are able to apply the work of Kontsevich and Soibelman, 
then Conjecture~\ref{conj:DTPT} follows from
the method in this paper. 
T. Bridgeland~\cite{BrH} also recently 
gives a proof
 of Conjecture~\ref{conj:DTPT}
assuming the result of~\cite{K-S}. 
His method is different from ours, and 
does not use any notion of stability 
conditions.  
In~\cite{JS}, Joyce and Song also study wall-crossing formula 
of counting invariants
involving virtual classes. At the moment
the author writes the first version of this paper, their work 
applies to counting invariants of coherent sheaves, and not 
to those of objects in the derived category. 
The only issue is the 
derived category version of~\cite[Theorem~5.3]{JS}, 
that is we need to show that the 
 local 
moduli space of objects in the derived category 
is described as a critical locus of some convergent function.
If the result of~\cite[Theorem~5.3]{JS}
 is extended to the case of the derived category, 
Conjecture~\ref{conj:DTPT} follows as well
from the method in this paper. 

After the author wrote the first version of this 
paper, 
it is announced that 
the above problem on the description of the local 
moduli space in the derived category is solved by 
Behrend and Getzler~\cite{BG}. 
Therefore we should now be able to give a complete proof
of Conjecture~\ref{conj:DTPT}. 
In the Appendix, we will give a proof of Conjecture~\ref{conj:DTPT}
using the result of~\cite{BG}. 

Finally we comment that Stoppa and Thomas~\cite{StTh}
investigate DT/PT correspondence via 
the wall-crossing of GIT stability. 
Then they show the same result of Theorem~\ref{thm:main}
applying 
 Joyce's theory, independently to our work. 
 It is remarkable that they do not use Joyce's 
 counting invariants of strictly semistable objects, 
 which will be introduced in
 Proposition-Definition~\ref{strict} in this paper.

\subsection{Content of the paper}
In Section~\ref{sec:compact}, 
we introduce the notion of weak stability conditions 
on triangulated categories, and study their 
general properties. 
In Section~\ref{sec:main}, we give a proof of 
Theorem~\ref{thm:main} assuming the result in the latter
sections
in a general setting. 
In Section~\ref{sec:walcro}, we give a general 
framework to discuss wall-crossing formula. 
In Section~\ref{sec:Wcro}, we establish the wall-crossing 
formula of generating series. 
In Section~\ref{sec:tech} and Section~\ref{sec:some}, we 
give the proofs of several technical lemmas. 

\subsection{Acknowledgement}
The author thanks Tom Bridgeland, 
Yunfeng Jiang,
Dominic Joyce, 
Kentaro Nagao and 
Richard Thomas 
for valuable discussions. 
He also thanks Kazushi Ueda for several comments on 
the manuscript, and Yan Soibelman for 
the comment on his work with Maxim Kontsevich. 
This work is supported by 
World Premier International 
Research Center Initiative
(WPI initiative), MEXT, Japan.
This work is partially supported by 
EPSRC grant EP/F038461/1.

\subsection{Notation and convention}
In this paper, all the varieties are 
defined over $\mathbb{C}$. 
For a triangulated category $\dD$, 
the shift functor is denoted by $[1]$. 
For a set of objects 
$S\subset \dD$, we denote by $\langle S \rangle_{\tr}\subset \dD$
the smallest triangulated subcategory of $\dD$ which 
contains $S$. 
Also we denote by $\langle S \rangle_{\ex}$ the 
smallest extension closed subcategory of $\dD$ which 
contains $S$. For an abelian category $\aA$ and a set of objects
$S\subset \aA$, the subcategory
$\langle S\rangle_{\ex} \subset \aA$
is also defined to be the smallest extension closed 
subcategory of $\aA$ which contains $S$. 
The abelian category of coherent sheaves is 
denoted by $\Coh(X)$. We say $F\in \Coh(X)$ is 
$d$-dimensional if its support is $d$-dimensional.

\section{Weak stability conditions on triangulated categories}
\label{sec:compact}
In this section, we introduce the notion 
of weak
stability conditions 
on triangulated categories, which generalizes 
Bridgeland's stability conditions~\cite{Brs1}. 
\subsection{Slicings}
Let $\dD$ be a triangulated category. Here 
we recall the notion of slicings on $\dD$ 
given in~\cite[Section~3]{Brs1}. 
\begin{defi}\label{def:slice}
\emph{
A \textit{slicing} on $\dD$ consists of a family of full subcategories 
$\{\pP(\phi)\}_{\phi \in \mathbb{R}}$, which satisfies the 
following. }
\begin{itemize}
\item \emph{For any $\phi \in \mathbb{R}$, we have 
$\pP(\phi)[1]=\pP(\phi+1)$.}
\emph{\item For $E_i \in \pP(\phi_i)$ with $\phi_1>\phi_2$, we have 
$\Hom(E_1, E_2)=0$. }
\emph{\item {\bf(Harder-Narasimhan filtration):}
For any non-zero object $E\in \dD$, we have the 
following collection of triangles:}
\begin{align}\label{Harder}
\xymatrix{
0=E_0 \ar[rr]  & &E_1 \ar[dl] \ar[rr] & & E_2 \ar[r]\ar[dl] & \cdots \ar[rr] & & E_n =E \ar[dl]\\
&  F_1 \ar[ul]^{[1]} & & F_2 \ar[ul]^{[1]}& & & F_n \ar[ul]^{[1]}&
}\end{align}
\emph{such that $F_j \in \pP (\phi _j)$ with $\phi _1 > \phi _2 > \cdots >\phi _n$. }
\end{itemize}
\end{defi} 
We also need an additional condition
called the \textit{local finiteness}. 
For an interval $I\subset \mathbb{R}$, the category $\pP(I)\subset \dD$
is defined to be 
$$\pP(I)=\langle \pP(\phi) : \phi \in I \rangle_{\ex} \subset \dD.$$
If $I=(a, b)$ with $b-a<1$, then the category $\pP(I)$ is 
a quasi-abelian category (cf.~\cite[Definition~4.1]{Brs1}). 
If we have a distinguished triangle 
$$A \stackrel{i}{\lr} B \stackrel{j}{\lr}C \lr A[1]$$
with $A, B, C\in \pP(I)$, we say $i$ is a \textit{strict 
monomorphism}, and $j$ is a \textit{strict epimorphism}. 
Then we say $\pP(I)$ is of \textit{finite length} if 
$\pP(I)$ is noetherian and artinian with respect to 
strict epimorphisms and strict monomorphisms respectively. 
(See~\cite[Section~4]{Brs1} for the detail.)
\begin{defi}\emph{
A slicing $\{\pP(\phi)\}_{\phi\in \mathbb{R}}$ is 
\textit{locally finite} if there 
exists $\eta>0$ such that 
for any $\phi \in \mathbb{R}$, 
the quasi-abelian category 
$\pP((\phi-\eta, \phi+\eta))$ is of finite length.}
\end{defi}
The set of locally finite slicings on $\dD$
 is denoted by $\Slice(\dD)$. 
 For $0\neq E\in \dD$ and $\pP\in \Slice(\dD)$, we 
 set $\phi_{\pP}^{+}(E)=\phi_1$
 and $\phi_{\pP}^{-}(E)=\phi_n$, where 
 $\phi_i$ are given by the 
 last condition of Definition~\ref{def:slice}. 
There is a generalized metric on $\Slice(\dD)$, 
given by  
\begin{align}\label{metSli}
d(\pP, \qQ)=
\sup_{0\neq E\in \dD}\left\{ 
\lvert \phi_{\pP}^{-}(E)-\phi_{\qQ}^{-}(E)\rvert, 
\lvert \phi_{\pP}^{+}(E)-\phi_{\qQ}^{+}(E)\rvert
\right\}\in [0, \infty], 
\end{align}
for $\pP, \qQ \in \Slice(\dD)$. 
It is shown in~\cite[Section~6]{Brs1} that 
$d(\pP, \qQ)=d(\qQ, \pP)$ and $d(\pP, \qQ)=0$ implies 
$\pP=\qQ$.  
\subsection{Weak stability conditions}
For a triangulated category $\dD$, let $K(\dD)$ be 
the Grothendieck 
group of $\dD$. We fix a finitely generated free abelian 
group $\Gamma$ together with a group homomorphism, 
$$\cl \colon K(\dD) \to \Gamma.$$
We also fix a filtration of 
$\Gamma$, 
\begin{align}\label{filt}
0\subsetneq \Gamma_0 \subsetneq \Gamma_1 \cdots \subsetneq 
\Gamma_N=\Gamma,
\end{align}
such that each subquotient 
$$\mathbb{H}_i\cneq \Gamma_i/\Gamma_{i-1}, \quad (0\le i\le N)$$
is a free abelian group. 
We set 
$\mathbb{H}_i^{\vee}\cneq \Hom_{\mathbb{Z}}(\mathbb{H}_i, \mathbb{C})$, 
and fix a norm $\lVert \ast \rVert_i$ on 
$\mathbb{H}_i \otimes_{\mathbb{Z}}\mathbb{R}$. 
 Given an element 
 $$Z=\{Z_i\}_{i=0}^{N} \in \prod_{i=0}^{N}\mathbb{H}_i^{\vee}, $$
 we define a map $Z\colon \Gamma \to \mathbb{C}$ as follows. 
 For $v\in \Gamma$, there is unique $0 \le m\le N$ 
 such that $v\in \Gamma_m\setminus \Gamma_{m-1}$. 
 Here we set $\Gamma_{-1}=\emptyset$. 
 Then $Z(v)$ is defined by 
 $$Z(v)\cneq Z_m([v]) \in \mathbb{C},$$
 where $[v]$ is a class of $v$ in 
 $\mathbb{H}_m$. 
  Using such $0 \le m\le N$, 
  the following map is also defined, 
  \begin{align}\label{norm}
  \lVert \ast \rVert \colon \Gamma \ni v 
  \mapsto \lVert v \rVert_{m} \in \mathbb{R}.
  \end{align}
  Below we often write 
$\cl(E) \in \Gamma$
 as $E\in \Gamma$ when there is no confusion. 

\begin{defi}\label{def:stab}
\emph{
We define the set $\Stab_{\Gamma_{\bullet}}(\dD)$ 
to be pairs $(Z, \pP)$, 
\begin{align}\label{pair0}
Z \in \prod_{i=0}^{N}\mathbb{H}_i^{\vee}, 
\quad \pP\in \Slice(\dD),
\end{align}
which satisfy the following axiom. }
\begin{itemize}
\item \emph{For any non-zero $E\in \pP(\phi)$, we have} 
\begin{align}\label{phase}
Z(E) \in \mathbb{R}_{>0}\exp(i\pi \phi).
\end{align}
\item \emph{{\bf(Support property):}
There is a constant $C>0$ such that for any 
non-zero $E\in \bigcup_{\phi\in \mathbb{R}}\pP(\phi)$, we have}
\begin{align}\label{support}
\lVert E \rVert \le C\lvert Z(E) \rvert.
\end{align}
\end{itemize}
\end{defi}
\begin{rmk}
If $N=0$ in (\ref{filt}), the set $\Stab_{\Gamma_{\bullet}}(\dD)$
coincides with the set of stability conditions
on $\dD$ introduced by Bridgeland~\cite{Brs1}, 
satisfying the support property. The support property 
is introduced by Kontsevich and Soibelman~\cite{K-S}
to refine 
Bridgeland stability and 
introduce the notion of stability data. We will use this 
property to show Theorem~\ref{thm:stab} below. 
\end{rmk}
\begin{rmk}
When $N=0$ in (\ref{filt}), the
 local finiteness condition automatically 
 follows if the support property is satisfied. 
 However for $N>0$, it seems that there is no reason 
 to conclude the local finiteness from the support
 property. 
\end{rmk}
We call an element of $\Stab_{\Gamma_{\bullet}}(\dD)$ 
a \textit{weak stability condition} on $\dD$. 
Although it is difficult to find a Bridgeland
stability on the derived category 
of coherent sheaves on algebraic varieties, it is 
rather easier to find a weak one as we see below. 
\begin{exam}\label{ex:pure}
\emph{
Let $X$ be a smooth projective variety of $\dim X=d$, 
$\omega$ an ample divisor on $X$, 
and $\dD=D^b(\Coh(X))$. 
We set $\Gamma$ to be the image of the chern character map, 
$$\cl\cneq \ch \colon K(\dD) \twoheadrightarrow
\Gamma\subset  H^{\ast}(X, \mathbb{Q}).$$
We choose a filtration (\ref{filt}) as 
$$\Gamma_i\cneq \Gamma \cap H^{\ge 2d-2i}(X, \mathbb{Q}).$$
In this case 
we have $\mathbb{H}_i=\Gamma \cap H^{2d-2i}(X, \mathbb{Q})$.
 Choose 
$0<\phi_d<\phi_{d-1}<\cdots <\phi_0<1$ and set 
$Z_i \in \mathbb{H}_i^{\vee}$ to be 
\begin{align}\label{ex:z}
Z_i(v)=\exp(i\pi \phi_i)\int_{X} v\cdot \omega^{i} \in \mathbb{C}.
\end{align}
We define the slicing $\{\pP(\phi)\}_{\phi \in \mathbb{R}}$ as 
follows. For $0<\phi \le 1$ with $\phi\neq \phi_i$ for 
any $i$, we set
$\pP(\phi)=\emptyset$. For $\phi=\phi_i$, set
$$\pP(\phi_i)=\{ E\in \Coh(X) : E\mbox{\rm{ is pure of }}
\dim \Supp(E)=i\}.$$
Other $\pP(\phi)$ for $\phi \in \mathbb{R}$
 is determined by the 
 first condition of Definition~\ref{def:slice}. 
 It is easy to check that 
$(\{Z_i\}_{i=0}^{d}, \pP)$ is an element of 
$\Stab_{\Gamma_{\bullet}}(\dD)$. }
\end{exam}
In what follows, we use the following notation. 
For $\sigma=(Z, \pP) \in \Stab_{\Gamma_{\bullet}}(\dD)$
and an interval $I\subset \mathbb{R}$, we set
\begin{align}\label{CI}
C_{\sigma}(I)\cneq 
\Imm (\cl \colon \pP(I) \to \Gamma)\subset \Gamma.
\end{align}
If $I\subset (a, a+1]$ for some $a\in \mathbb{R}$, we 
can define the \textit{phase} of 
$v\in C_{\sigma}(I)$ by, 
\begin{align}\label{def:phase}
\phi_{\sigma}(v)=\frac{1}{\pi}\arg Z(v) \in I.
\end{align}
\subsection{Constructions via t-structures} 
In this paragraph, we give another way of 
constructing elements of $\Stab_{\Gamma_{\bullet}}(\dD)$, 
using the notion of bounded t-structures. 
The readers can refer~\cite[Section~3]{Brs1} for
the notion of  
bounded t-structures, and their hearts.  
\begin{defi}\label{def:wefu}
\emph{
Let $\aA\subset \dD$ be the heart of a bounded t-structure on 
$\dD$. We say $Z \in \prod_{i=0}^{N}\mathbb{H}_i^{\vee}$
is a \textit{weak stability function} on $\aA$ if for any non-zero $E\in \aA$, 
we have} 
\begin{align}\label{def:weak}
Z(E) \in \mathfrak{H}\cneq \{ r\exp(i\pi \phi) : r>0, \ 0<\phi \le 1\}.
\end{align}
\end{defi}
By (\ref{def:weak}), we can uniquely 
determine the argument,
$$\arg Z(E) \in (0, \pi],$$
for any $0\neq E\in \aA$. 
For an exact sequence $0\to F\to E \to G\to 0$
in $\aA$, 
one of the
following equalities holds.
\begin{align*}
&\arg Z(F) \le \arg Z(E) \le \arg Z(G), \\
&\arg Z(F) \ge \arg Z(E) \ge \arg Z(G).
\end{align*}
\begin{rmk}\label{=<}
When $N=0$, a weak stability function 
coincides with a stability function introduced 
in~\cite[Definition~2.1]{Brs1}. 
In this case 
we have one of the following inequalities
\begin{align*}
&\arg Z(F) < \arg Z(E) < \arg Z(G), \\
&\arg Z(F) > \arg Z(E) > \arg Z(G), \\
&\arg Z(F)=\arg Z(E)=\arg Z(G),
\end{align*}
so 
an inequality such as $\arg Z(F)<\arg Z(E)=\arg Z(G)$
does not occur.
On the other hand we might have such an inequality
when $N>0$, 
and such a 
 function 
 determines a weak stability condition
in the sense of~\cite[Definition~4.1]{Joy3}. 
\end{rmk}
\begin{defi}\emph{
Let $Z\in \prod_{i=0}^{N}\mathbb{H}_i^{\vee}$
be a weak stability function on $\aA$. We say 
$0\neq E \in \aA$ is $Z$-\textit{semistable} (resp. \textit{stable}) if 
for any exact sequence $0 \to F \to E \to G\to 0$ we have} 
\begin{align}\label{argg}
\arg Z(F)\le \arg Z(G), \quad (\mbox{\rm{resp}.~}\arg Z(F)<\arg Z(G).)
\end{align}
\end{defi}
\begin{rmk}
If $N=0$, the
 condition (\ref{argg}) is equivalent to 
$\arg Z(F) \le \arg Z(E)$, (resp.~$\arg Z(F)<\arg Z(E)$.)
However for $N>0$, the condition (\ref{argg}) is not equivalent to 
the above condition, since we may have 
$\arg Z(F)<\arg Z(E)=\arg Z(G)$, as in Remark~\ref{=<}.
\end{rmk}
The notion of Harder-Narasimhan filtration is defined 
in a similar way to usual stability conditions. 
\begin{defi}\emph{
Let $Z\in \prod_{i=0}^{N}\mathbb{H}_i^{\vee}$
be a weak stability function on $\aA$.
A \textit{Harder-Narasimhan filtration} of an object $E\in \aA$ 
is a filtration 
$$0=E_0 \subset E_1 \subset \cdots \subset E_{k-1}\subset E_k=E, $$
such that each subquotient $F_j=E_j/E_{j-1}$ is 
$Z$-semistable with 
$$\arg Z(F_1)>\arg Z(F_2)>\cdots >\arg Z(F_k).$$
A weak stability function $Z$
is said to have the \textit{Harder-Narasimhan property} if any 
object $E\in \aA$ has a Harder-Narasimhan filtration.} 
\end{defi}
The following proposition is 
an analogue of~\cite[Proposition~2.4]{Brs1}. 
\begin{prop}\label{suHN}
Let $Z\in \prod_{i=0}^{N}\mathbb{H}_i^{\vee}$
be a weak stability function on $\aA$.
Suppose that the following chain conditions are satisfied. 

(a) There are no infinite sequences of subobjects in $\aA$, 
$$\cdots \subset E_{j+1} \subset E_j \subset \cdots \subset E_2 \subset E_1$$
with $\arg Z(E_{j+1})>\arg Z(E_j/E_{j+1})$ for all $j$. 

(b) There are no infinite sequences of quotients in $\aA$, 
$$E_1 \twoheadrightarrow E_2 \twoheadrightarrow \cdots \twoheadrightarrow 
E_j \stackrel{\pi_j}{\twoheadrightarrow}
 E_{j+1} \twoheadrightarrow \cdots$$
with $\arg Z(\Ker \pi_j)>\arg Z(E_{j+1})$ for all $j$. 

Then $Z$ has the Harder-Narasimhan property. 
\end{prop}
\begin{proof}
Although our stability condition is a weak one, 
the same proof of~\cite[Proposition~2.4]{Brs1} still 
works. 
Also see the proof of~\cite[Theorem~4.4]{Joy3}. 
\end{proof} 
The following proposition is an analogue of~\cite[Proposition~5.3]{Brs1}, 
which relates weak stability conditions and weak stability functions on
the hearts of
 bounded t-structures. 
\begin{prop}\label{prop:corr}
Giving a pair
$(Z, \pP)$, where $Z\in \prod_{i=0}^{N}\mathbb{H}_i^{\vee}$
and $\pP$ is a slicing, satisfying (\ref{phase})
 is equivalent to 
giving a bounded t-structure on $\dD$ and a weak 
stability function on its heart with the Harder-Narasimhan 
property. 
\end{prop}
\begin{proof}
The proof is same as in~\cite[Proposition~5.3]{Brs1},
so we just describe how to give the correspondence. 
Given $Z\in \prod_{i=0}^{N}\mathbb{H}_i^{\vee}$
and a slicing $\{\pP(\phi)\}_{\phi\in \mathbb{R}}$
  satisfying (\ref{phase}), 
 the category 
 $$\aA=\pP((0, 1]),$$
  is the heart of a bounded 
t-structure and $Z$ is a weak 
stability function on $\aA$. Conversely suppose 
that $\aA \subset \dD$ is the heart of a bounded 
t-structure on $\dD$ and $Z$ 
is a weak stability function on it. 
For $0<\phi \le 1$, let $\pP(\phi)$ 
be the full additive subcategory of $\aA$, 
defined by 
$$\pP(\phi)=\left\{ E\in \aA : \begin{array}{l}
 E\mbox{ is }Z\mbox{-semistable with }\\
 Z(E) \in \mathbb{R}_{>0}\exp(i\pi \phi)
 \end{array}\right\}.
$$
The subcategory
 $\pP(\phi)$ for all $\phi \in \mathbb{R}$ is determined 
by the first condition of Definition~\ref{def:slice}. 
By the Harder-Narasimhan 
property of $Z$, $\pP$ is a slicing on $\dD$. 
\end{proof}
Below we write an element of $\Stab_{\Gamma_{\bullet}}(\dD)$
as $(Z, \pP)$ with $\pP\in \Slice(\dD)$, 
or 
$(Z, \aA)$ with $\aA\subset \dD$ the heart of a bounded
t-structure on $\dD$. 
The above proposition enables us to produce 
more examples of stability conditions. 
\begin{exam}\label{ex:A}
\emph{
(i)
Let $X$ be a $d$-dimensional smooth projective variety, 
$\dD=D^b(\Coh(X))$ and $\aA=\Coh(X)\subset \dD$. 
Then $Z=\{Z_i\}_{i=0}^{d}$ defined by (\ref{ex:z})
is a weak stability function on $\aA$. 
An object $E\in \aA$ is $Z$-semistable if and only if 
$E$ is a pure sheaf, thus $\{Z_i\}_{i=0}^{d}$
satisfies the Harder-Narasimhan property. 
In this way, we can recover 
the slicing $\{\pP(\phi)\}_{\phi \in \mathbb{R}}$ 
given in Example~\ref{ex:pure}.}

\emph{(ii)
Let $A$ be a finite dimensional $\mathbb{C}$-algebra, 
$\aA=\modu A$ the category of finitely generated right 
$A$-modules and $\dD=D^b(\aA)$. 
There is a finite number of simple objects 
$S_0, S_1, \cdots, S_N \in \aA$ such that 
$$K(\dD)=\bigoplus_{j=0}^{N}\mathbb{Z}[S_j].$$ 
Let $\Gamma=K(\dD)$ and $\cl \colon K(\dD) \to 
\Gamma$ the identity map. 
We choose the filtration (\ref{filt})
to be $\Gamma_j=\bigoplus_{a=0}^{j}\mathbb{Z}[S_a]$, 
hence $\mathbb{H}_j=\mathbb{Z}[S_j]$. 
Choose $0<\phi_j \le 1$ for 
$0\le j\le N$ 
 and set 
$Z_{j} \in \mathbb{H}_j^{\vee}$
to be 
$$Z_{j}(r[S_j])=r\exp(i\pi \phi_j). $$
Then $\{Z_{j}\}_{j=0}^{N}$ 
is a weak stability function on $\aA$. The corresponding 
pair $(\{Z_{j}\}_{j=0}^{N}, \pP)$ via Proposition~\ref{prop:corr}
gives an element 
of $\Stab_{\Gamma_{\bullet}}(\dD)$.} 
\end{exam}
\subsection{The space of weak stability conditions}
There is the 
inclusion, 
\begin{align}\label{inclu}
\Stab_{\Gamma_{\bullet}}(\dD) \subset \Slice(\dD)\times
\prod_{i=0}^{N}\mathbb{H}_i^{\vee}.
\end{align}
The generalized metric defined by (\ref{metSli}) 
induces a topology on $\Slice(\dD)$, and
we equip the set $\prod_{i=0}^{N}\mathbb{H}_i^{\vee}$
 with an usual Euclid topology. 
Thus we obtain the induced topology on $\Stab_{\Gamma_{\bullet}}(\dD)$
via the inclusion (\ref{inclu}). 
The following theorem is a generalization of~\cite[Theorem~1.2]{Brs1}, 
which makes each connected component of 
$\Stab_{\Gamma_{\bullet}}(\dD)$ a complex manifold. 
\begin{thm}\label{thm:stab}
The map
$$\Pi\colon \Stab_{\Gamma_{\bullet}}(\dD)
\ni(Z, \pP) \longmapsto Z 
\in \prod_{i=0}^{N}\mathbb{H}_i^{\vee}, $$
is a local homeomorphism. In particular 
each connected component of $\Stab_{\Gamma_{\bullet}}(\dD)$ 
is a complex manifold. 
\end{thm}
\begin{proof}
The proof is almost same as in~\cite[Theorem~7.1]{Brs1}. 
We give the outline of the proof in Section~\ref{sec:some}.
\end{proof}
\begin{rmk}\label{rmk:closed}
For $\sigma=(Z, \pP)\in \Stab_{\Gamma_{\bullet}}(\dD)$
and $0\neq E\in \dD$, we 
set $\phi_{\sigma}^{\pm}(E)\cneq \phi_{\pP}^{\pm}(E)$. 
By the definition of the generalized metric (\ref{metSli}), 
the maps 
$$\phi_{\ast}^{\pm}(E)\colon 
\Stab_{\Gamma_{\bullet}}(\dD) \to \mathbb{R}, $$
are continuous. In particular, the subset of 
$\sigma\in\Stab_{\Gamma_{\bullet}}(\dD)$ 
in which $E$ is 
semistable is 
a closed subset.  
\end{rmk}
Later on we will need the 
following lemma, 
which relates a family of points in $\prod_{i=0}^{N}\mathbb{H}_i^{\vee}$
and points in $\Stab_{\Gamma_{\bullet}}(\dD)$. 
The proof will be given in Section~\ref{sec:some}.
\begin{lem}\label{lem:cont}
Let 
$$(0, 1) \ni t \longmapsto Z_t \in 
\prod_{i=0}^{N}\mathbb{H}_i^{\vee},$$
be a continuous map, and 
$\aA \subset \dD$ the heart of a bounded t-structure 
on $\dD$. 
Suppose
 that $\sigma_t=(Z_t, \aA)$ determine points 
in $\Stab_{\Gamma_{\bullet}}(\dD)$. Then 
$\{\sigma_t\}_{t\in (0, 1)}$ is a
continuous family in $\Stab_{\Gamma_{\bullet}}(\dD)$. 
\end{lem}

\subsection{Group action}
Similarly to Bridgeland's stability conditions, 
the space $\Stab_{\Gamma_{\bullet}}(\dD)$ carries 
a group action of 
$\widetilde{\GL}^{+}(2, \mathbb{R})$, which 
is a universal covering space of $\GL^{+}(2, \mathbb{R})$.
Although we do not need this group action
in this paper, it seems 
worth putting it here as 
an analogue of~\cite[Lemma~8.2]{Brs1}. 
\begin{lem}\label{lem:group}
The space 
$\Stab_{\Gamma_{\bullet}}(\dD)$ carries a right 
action of the group $\widetilde{\GL}^{+}(2, \mathbb{R})$. 
\end{lem}
\begin{proof}
Note that the group $\widetilde{\GL}^{+}(2, \mathbb{R})$
is identified with the set of pairs $(T, f)$, where 
$f\colon \mathbb{R} \to \mathbb{R}$ is an increasing map 
with $f(x+1)=f(x)+1$, and 
$T\in \GL^{+}(2, \mathbb{R})$ such that the induced maps on 
$S^1=\mathbb{R}/2\mathbb{Z}=(\mathbb{R}^2\setminus\{0\})/\mathbb{R}_{>0}$
are the same. 
Given $\sigma=(\{Z_i\}_{i=0}^{N}, \pP) \in \Stab_{\Gamma_{\bullet}}(\dD)$ 
and $(T, f) \in \widetilde{\GL}^{+}(2, \mathbb{R})$, we 
set $Z_i'=T^{-1}\circ Z_i$ and $\pP'(\phi)=\pP(f(\phi))$. 
Then $\sigma'=(\{Z_i'\}_{i=0}^{N}, \pP')$
gives an element of $\Stab_{\Gamma_{\bullet}}(\dD)$, and 
the right action of $\widetilde{\GL}^{+}(2, \mathbb{R})$ is 
given in this way. 
\end{proof}
We give an example on the global structure 
of $\Stab_{\Gamma_{\bullet}}(\dD)$. 
\begin{exam}\emph{
Let $C$ be an elliptic curve, and $\dD=D^b(\Coh(C))$. 
We set $\Gamma$ and the filtration $\Gamma_{\bullet}$
as in Example~\ref{ex:pure}. 
In this case, 
the same proof of~\cite[Theorem~9.1]{Brs1} shows that
the action of $\widetilde{\GL}^{+}(2, \mathbb{R})$
on $\Stab_{\Gamma_{\bullet}}(\dD)$ is free and transitive. 
Hence we have} 
$$\Stab_{\Gamma_{\bullet}}(\dD) \cong \widetilde{\GL}^{+}(2, \mathbb{R}).$$
\end{exam}

\begin{rmk}
There is a close relationship between weak stability conditions 
and polynomial stability conditions introduced by Bayer~\cite{Bay}. 
Let us fix an isomorphism
\begin{align}\label{rmk:isom}
\Gamma \cong \bigoplus_{i=0}^{N}\mathbb{H}_i, 
\end{align}
and take a pair of the heart of a bounded t-structure 
$\aA \subset \dD$ and $Z=\{Z_i\}_{i=0}^{N} \in 
\prod_{i=0}^{N}\mathbb{H}_i^{\vee}$, satisfying (\ref{def:weak}).
Then the polynomial function $\widetilde{Z}_{m}\colon \Gamma \to \mathbb{C}$
given by 
$$\widetilde{Z}_{m}(v)=\sum_{i=0}^{N} m^i Z_i(v_i), $$
where $v_i \in \mathbb{H}_i$ is the $i$-th component of $v$, 
satisfies $\widetilde{Z}_m(v)\in \mathfrak{H}$ for $m\gg 0$. 
Hence the pair $(\widetilde{Z}_m, \aA)$ gives a 
polynomial stability condition, if the Harder-Narasimhan property is 
satisfied. However the set of 
$Z$-(semi)stable objects and that of $\widetilde{Z}_m$-(semi)stable
objects are different. It is easy to see that 
$$Z\mbox{-stable} \Rightarrow \widetilde{Z}_m\mbox{-stable}
\Rightarrow \widetilde{Z}_{m}\mbox{-semistable} \Rightarrow 
Z\mbox{-semistable}.$$
Therefore the notion of 
weak stability conditions is more coarse than that of
polynomial stability conditions. Roughly speaking, 
a polynomial stability condition is an analogue of Gieseker stability, 
and a weak stability condition is an analogue of 
$\mu$-stability. 
\end{rmk}
\begin{rmk}
It seems that $\Stab_{\Gamma_{\bullet}}(\dD)$ is 
a space of 
limiting degeneration points of the 
usual space of stability conditions $\Stab(\dD)$. 
Under the isomorphism (\ref{rmk:isom}), 
the multiplicative group $\mathbb{R}_{>0}$ acts on 
$\Gamma^{\vee}\cneq \Hom(\Gamma, \mathbb{C})
\cong \prod_{i=0}^{N}\mathbb{H}_i^{\vee}$ via 
$$t\cdot (Z_0, Z_1, \cdots, Z_N)=(Z_0, tZ_1, \cdots, t^N Z_N).$$
The above action lifts to an action on $\Stab_{\Gamma_{\bullet}}(\dD)$
via $(Z, \pP) \mapsto (t\cdot Z, \pP)$. Presumably 
there is a natural topology on the set, 
$$\overline{\Stab}_{\Gamma_{\bullet}}(\dD)
=\Stab(\dD) \coprod 
\left( \Stab_{\Gamma_{\bullet}}(\dD)/
\mathbb{R}_{>0}\right),$$
which makes $\overline{\Stab}_{\Gamma_{\bullet}}(\dD)$ 
a complex manifold with real 
codimension one boundary $\Stab_{\Gamma_{\bullet}}(\dD)/
\mathbb{R}_{>0}$.  
A choice of filtrations $\Gamma_{\bullet}$ should 
corresponds to a choice of limiting directions. 
\end{rmk}

\section{Proof of the Main Theorem}\label{sec:main}
In what follows, we assume that $X$ is a smooth projective 
Calabi-Yau 3-fold over $\mathbb{C}$.
We set $\Coh_{\le 1}(X)$ to be
\begin{align}\label{coh1}
\Coh_{\le 1}(X) &\cneq \{E\in \Coh(X) : \dim \Supp(E) \le 1\}. 
\end{align}
In this section, we show how Theorem~\ref{thm:main}
is proved via wall-crossing phenomena in the 
space of weak stability conditions on 
the following triangulated category, 
\begin{align}\label{D}
\dD_X=\langle \oO_X, \Coh_{\le 1}(X) \rangle_{\tr} \subset D^b(\Coh(X)).
\end{align}
We will construct weak stability conditions on 
$\dD_X$, investigate corresponding stable objects, and 
show Theorem~\ref{thm:main}. 
In the proof of Theorem~\ref{thm:main}, we will use 
a wall-crossing formula which will be established 
under a general setting in Section~\ref{sec:walcro}. 
\subsection{Construction of a t-structure on $\dD_X$.}
We begin with constructing a t-structure on $\dD_X$. 
First we recall the notion of torsion pairs
and tilting. 
\begin{defi}\emph{\bf{\cite{HRS}}}
\emph{Let $\aA$ be an abelian category, and 
$(\tT, \fF)$ a pair of subcategories of $\aA$. 
We say $(\tT, \fF)$ is a \textit{torsion pair}
if the following conditions hold.} 
\begin{itemize}
\item \emph{$\Hom(T, F)=0$ for any $T\in \tT$ and 
$F\in \fF$.}
\item \emph{Any object $E\in \aA$
fits into an exact sequence, 
\begin{align}\label{fits}
0 \lr T \lr E \lr F \lr 0, 
\end{align}
with $T\in \tT$ and $F\in \fF$.}
\end{itemize}
\end{defi}
Given a torsion pair $(\tT, \fF)$ on 
$\aA$, its \textit{tilting} is 
defined by 
\begin{align}\label{dag}
\aA^{\dag}\cneq \left\{ E\in D^b(\aA) : \begin{array}{l}
\hH^{-1}(E) \in \fF, \ \hH^{0}(E) \in \tT, \\
\hH^i(E)=0 \mbox{ for }i\notin \{-1, 0\}.
\end{array}
\right\},
\end{align}
i.e. $\aA^{\dag}=\langle \fF[1], \tT \rangle_{\ex}$ in 
$D^b(\aA)$. It is known that $\aA^{\dag}$ is the 
heart of a bounded t-structure on $D^b(\aA)$. 
(cf.~\cite[Proposition~2.1]{HRS}.)

\hspace{2mm}

Let $\Coh_{\ge 2}(X)$ be the subcategory of $\Coh(X)$, 
\begin{align}\label{oc}
\Coh_{\ge 2}(X)\cneq \{ E\in \Coh(X) : \Hom(\Coh_{\le 1}(X), E)=0\}.
\end{align}
It is easy to see that the pair 
\begin{align}\label{torpair}
(\Coh_{\le 1}(X), \Coh_{\ge 2}(X)),
\end{align} is a torsion pair on $\Coh(X)$. 
\begin{defi}\label{Coh}\emph{
We define the abelian category $\Coh^{\dag}(X)$ to be
the tilting with respect to (\ref{torpair}), i.e. 
$$\Coh^{\dag}(X)=\langle \Coh_{\ge 2}(X)[1], \Coh_{\le 1}(X)\rangle_{\ex}.$$}
\end{defi}
\begin{rmk}
The category $\Coh^{\dag}(X)$ is one of the hearts of 
 perverse t-structures introduced 
by Bezrukavnikov~\cite{Bez} and Kashiwara~\cite{Kashi}. 
\end{rmk}
\begin{rmk}\label{clo}
It is easy to see that the subcategory 
$\Coh_{\le 1}(X)\subset \Coh^{\dag}(X)$ is closed under 
subobjects and quotients in $\Coh^{\dag}(X)$. 
\end{rmk}
The above construction induces a t-structure on $\dD_X$. 
\begin{lem}\label{percoh}
The intersection $\aA_X=\dD_X \cap \Coh^{\dag}(X)[-1]$
in $D^b(\Coh(X))$
is the heart of a bounded t-structure on $\dD_X$, 
and $\aA_X$ is written as 
\begin{align}\label{tst1}
\aA_{X}=\langle \oO_X, \Coh_{\le 1}(X)[-1]\rangle_{\ex}.
\end{align}
\end{lem}
\begin{proof}
Note that we have 
\begin{align}\label{coh[-1]}
D^b(\Coh_{\le 1}(X)) \cap \Coh^{\dag}(X)[-1]=\Coh_{\le 1}(X)[-1],
\end{align}
 hence (\ref{coh[-1]}) is 
the heart of a bounded t-structure on $D^b(\Coh_{\le 1}(X))$.  
For $F\in \Coh_{\le 1}(X)$, we have 
$\Hom(\oO_X, F[-1])=0$ and 
\begin{align*}
\Hom(F[-1], \oO_X)&\cong H^2(X, F)^{\vee} \\
&= 0,
\end{align*}
by the Serre duality. 
Then the result follows by 
setting $\dD=D^b(\Coh(X))$, $\dD'=D^b(\Coh_{\le 1}(X))$, 
$\aA=\Coh^{\dag}(X)[-1]$ and $E=\oO_X$ in 
Proposition~\ref{t-str} below. 
\end{proof}
We have used the following proposition, 
whose proof will be given in Section~\ref{sec:tech}. 
\begin{prop}\label{t-str}
Let $\dD$ be a $\mathbb{C}$-linear triangulated category 
and $\aA \subset \dD$ the heart of a bounded t-structure on $\dD$. 
Take $E\in \aA$ with
$\End(E)=\mathbb{C}$ and a full triangulated subcategory
 $\dD'\subset \dD$, which satisfy 
 the following conditions. 
\begin{itemize}
 \item The category $\aA'\cneq \aA\cap \dD'$
is the heart of a bounded t-structure on $\dD'$, which 
is closed under subobjects and quotients in the 
abelian category $\aA$. 
\item For any object $F\in \aA'$, we have 
\begin{align}\label{cond}
\Hom(E, F)=\Hom(F, E)=0.
\end{align}
\end{itemize}
Let $\dD_E$ be the triangulated category, 
$$\dD_E \cneq \langle E, \dD' \rangle_{\tr} \subset \dD.$$
Then $\aA_E\cneq \dD_E \cap \aA$ is the heart 
of a bounded t-structure on $\dD_E$, 
which satisfies 
$$\aA_E=\langle E, \aA' \rangle_{\ex}.$$
\end{prop}

\begin{rmk}\label{Ox}
By Remark~\ref{clo}, 
the subcategory $\Coh_{\le 1}(X)[-1]\subset \aA_X$
is also closed under subobjects and quotient. 
In particular $\oO_x[-1] \in \aA_X$ is a simple 
object for any closed point $x\in X$. 
\end{rmk}

\subsection{Weak stability conditions on $\dD_X$}
In this paragraph, we construct weak stability conditions on 
$\dD_X$. 
 Let $N_{1}(X)$, $N^{1}(X)$ be the
 abelian groups of curves in $X$, divisors in $X$ respectively. 
 They are finite rank free abelian groups, and 
 there is a perfect pairing, 
 $$N_1(X)_{\mathbb{R}} \times N^1(X)_{\mathbb{R}} 
 \ni (C, D) \mapsto C\cdot D \in \mathbb{R}.$$
 We denote by $\NE(X)\subset N_1(X)$ the numerical 
 classes of effective curves, and 
 $A(X)\subset N^1(X)_{\mathbb{R}}$ the 
 ample cone. 
 We set 
 $$N_{\le 1}(X)=\mathbb{Z}\oplus N_1(X), $$
 and define $\Gamma$ to be
 $$\Gamma=N_{\le 1}(X) \oplus \mathbb{Z}.$$
 The group homomorphism $\cl \colon K(\dD_X) \to \Gamma$ is defined by 
$$\cl(E)=(\ch_3(E), \ch_2(E), \ch_0(E)).$$
By the definition of $\dD_X$, it is obvious that 
$\ch_{\bullet}(E)$ has integer coefficients
 thus $\cl$ is well-defined. 
 We denote by $\rk$ the projection onto the third factor, 
 $$\rk \colon \Gamma \ni (s, l, r) \mapsto r \in \mathbb{Z}.$$
 Let $\Gamma_0=\mathbb{Z}$, $\Gamma_1=N_{\le 1}(X)$
 and $\Gamma_2=\Gamma$. This defines a filtration $\Gamma_{\bullet}$, 
 \begin{align}\label{hook}
 \Gamma_0 \stackrel{i}{\hookrightarrow}
  \Gamma_1 \stackrel{j}{\hookrightarrow}
 \Gamma_2=\Gamma,
 \end{align}
  via $i(s)=(s, 0)$ and $j(s, l)=(s, l, 0)$. 
 We have 
 $$\mathbb{H}_0=\mathbb{Z}, \quad \mathbb{H}_1=N_1(X), \quad 
 \mathbb{H}_2=\mathbb{Z}, $$
 and there is a natural isomorphism, 
 \begin{align}\label{natural}
 \mathbb{C}\times N^1(X)_{\mathbb{C}}\times \mathbb{C}
 \stackrel{\cong}{\lr}\prod_{i=0}^2 \mathbb{H}_i^{\vee}.
 \end{align}
 For the elements, 
 $$z_0, z_1 \in \mathfrak{H} \mbox{ with }
 \arg z_i \in (\pi/2, \pi),  \quad \omega \in A(X), $$
 the data
 \begin{align}\label{daxi}
 \xi=(-z_0, -i\omega, z_1),
 \end{align} 
 associates the element $Z_{\xi} \in \prod_{i=0}^{2}
 \mathbb{H}_i^{\vee}$
 via the isomorphism (\ref{natural}).
 It is written as 
 \begin{align*}
 &Z_{0, \xi}\colon \mathbb{H}_0 \ni s  \longmapsto -s z_0, \\
 &Z_{1, \xi} \colon \mathbb{H}_1 \ni l \longmapsto -i\omega \cdot l, \\
 &Z_{2, \xi} \colon \mathbb{H}_2 \ni r \longmapsto r z_1.
 \end{align*}
 \begin{lem}\label{const:stab2}
  The pairs 
 \begin{align}\label{pair3}
 \sigma_{\xi}=(Z_{\xi}, \aA_X), \quad
 \xi \mbox{ is given by }(\ref{daxi}), \end{align}
 determine points in $\Stab_{\Gamma_{\bullet}}(\dD_X)$. 
 \end{lem}
 \begin{proof}
 We check that (\ref{def:weak}) holds for 
 any non-zero $E\in \aA_X$. 
 We write $\cl(E)=(-n, -\beta, r)$ for 
 $n\in \mathbb{Z}$, $\beta \in N_1(X)$ and $r\in \mathbb{Z}$. 
 By the description (\ref{tst1}), 
 we have either 
 $$r>0, \quad \mbox{or} \quad r=0, \ \beta\in \NE(X), \quad 
 \mbox{or} \quad r=\beta=0, \ n>0.$$
 Then (\ref{def:weak}) follows by 
 our construction of $Z_{\xi}$.  
 The proofs to check other properties, i.e. Harder-Narasimhan 
 property, support property, locally finiteness
 are straightforward. We give the proof 
 in Section~\ref{sec:tech}. 
 (The condition $\arg z_0>\pi/2$ will be
 required to show the local finiteness, and 
 $\arg z_1>\pi/2$ will be required in 
 Lemma~\ref{firstshow} below.) 
 \end{proof}
  We define the subspace $\vV_{X}\subset \Stab_{\Gamma_{\bullet}}(\dD_X)$
 as follows. 
 \begin{defi}\label{def:vV}
 \emph{
 We define $\vV_{X}\subset \Stab_{\Gamma_{\bullet}}(\dD_X)$ to be}
 \begin{align}\label{vV}
 \vV_X \cneq \{\sigma_{\xi} : 
 \sigma_{\xi} \emph{ is given by }(\ref{pair3}). 
 \}.
 \end{align}
 \end{defi}
 By Lemma~\ref{lem:cont}, 
 the map $\xi \mapsto \sigma_{\xi}$ is continuous. 
 In particular $\vV_X$ is a connected subspace.

  \subsection{Semistable objects of rank one}
 In this paragraph, we study 
 semistable objects
 in $\aA_X$ of rank one. 
 We first recall the notion of stable pairs. 
 \begin{defi}\emph{
 A pair $(F, s)$ is a \textit{stable pair} if it
 satisfies the following conditions.} 
 \begin{itemize}
 \item \emph{$F\in \Coh_{\le 1}(X)$ is a pure sheaf. i.e. 
 there is no 0-dimensional subsheaf $Q\subset F$. }
 \item \emph{$s\colon \oO_X \to F$ is a morphism with 
 0-dimensional cokernel.} 
 \item \emph{As a convention, we also call the pair $(0, 0)$
 a stable pair.}  
 \end{itemize}
 \end{defi}
 We have the following lemma.   
 \begin{lem}\label{firstshow}
 (i)
  Take $\sigma_{\xi}=(Z_{\xi}, \pP_{\xi})\in\vV_X$
  with $\pP_{\xi} \in \Slice(\dD_X)$, and 
 an object $E\in \pP_{\xi}((1/2, 1])$ 
 satisfying $\rk(E)=1$. 
 Then there is an exact sequence 
 in $\aA_X$, 
 \begin{align}\label{useseq}
 0\lr I_C \lr E \lr Q[-1] \lr 0,
 \end{align}
 where $I_C$ is an ideal sheaf of 1-dimensional 
 subscheme $C\subset X$ and 
 $Q$ is a 0-dimensional sheaf. 
 
 (ii) An object $E\in \aA_X$ fits into 
 a sequence (\ref{useseq}) if and only if 
 $E$ is isomorphic to a two-term 
 complex 
 \begin{align}\label{term}
 \cdots \to 0\to \oO_X \stackrel{s}{\to} F \to 0 \to \cdots,
 \end{align} 
 with $F\in \Coh_{\le 1}(X)$ and $s$ has 0-dimensional 
 cokernel. Here $\oO_X$ is located in degree zero and 
 $F$ is in degree one. 
 
 (iii) Let $E=(\oO_X \stackrel{s}{\to}F)$ be a
 two-term complex as in (\ref{term}). 
 Then $\Hom(\oO_x[-1], E)=0$ for all $x\in X$
 if and only if $(F, s)$ is a stable pair. 
 \end{lem}
 \begin{proof}
 (i) Since $E\in \langle \oO_X, \Coh_{\le 1}(X)[-1]\rangle_{\ex}$ 
 and $\rk(E)=1$, there is a filtration in $\aA_X$, 
 $$0=E_{-1}\subset E_0 \subset E_1 \subset E_2=E,$$
 such that for $F_i=E_{i}/E_{i-1}$, we have 
 $$F_0, F_2 \in \Coh_{\le 1}(X)[-1], \quad 
 F_1=\oO_X.$$
 Suppose that $F_2$ is 1-dimensional, 
 and let 
 $F_3\in \Coh_{\le 1}(X)[-1]$ be the object 
 such that 
 $F_3[1]\subset F_2[1]$ is the maximum 0-dimensional 
 subsheaf of $F_2[1]\in \Coh_{\le 1}(X)$. We have the surjections in $\aA_X$, 
 $$E\twoheadrightarrow F_2 \twoheadrightarrow F_2/F_3.$$
 On the other hand, 
 it is easy to see that 
 $$\pP_{\xi}(1/2)=\{F[-1]: F\mbox{ is a pure 1-dimensional 
 sheaf } \}, $$
 by noting Remark~\ref{Ox}. 
 Therefore we have $\Hom(E, F_2/F_3)=0$ by 
 the second condition of Definition~\ref{def:slice}. 
 This is a contradiction, hence $F_2$ is 0-dimensional.  
 This implies 
 the existence of the sequence (\ref{useseq}). 
 
 (ii) Obviously a two-term complex (\ref{term}) 
 fits into a sequence (\ref{useseq}). 
 Conversely let $E\in \aA_X$ be an object
 which fits into (\ref{useseq}). 
 The composition $Q[-2] \to I_C \hookrightarrow \oO_X$
 becomes a zero map since
 \begin{align*}\Hom(Q[-2], \oO_X) &\cong H^1(X, Q)^{\vee} \\
 &=0.
 \end{align*}
 Therefore $I_C \hookrightarrow \oO_X$ factorizes 
 via $I_C \to E \to \oO_X$. 
 Taking the cone, we obtain the distinguished triangle
 $$E \lr \oO_X \stackrel{s}{\lr} F.$$
 Since $F$ fits into the distinguished triangle
 $\oO_C \to F \to Q$, 
 we have $F\in \Coh_{\le 1}(X)$. 
 Also the cokernel of $s$ is isomorphic to 
 $Q$, hence it is 0-dimensional. 
 
 (iii) We have the exact sequence in $\aA_X$, 
 $$0 \lr F[-1] \lr E \lr \oO_X \lr 0.$$
 Applying $\Hom(\oO_x[-1], \ast)$
 to the above sequence yields, 
 $$\Hom(\oO_x, F)\cong \Hom(\oO_x[-1], E).$$
 Hence $\Hom(\oO_x[-1], E)=0$ for all
 $x\in X$ is equivalent to that $F$ is pure. 
 \end{proof}
 For $v\in \Gamma$ and $\sigma \in \vV_X$, 
 we set
 $$M^v(\sigma)\cneq \{ 
 E\in \aA_X : E\mbox{ is }\sigma\mbox{-semistable with }
 \cl(E)=v\}.$$
 The above set of object is described as follows.
 \begin{prop}\label{prop:DTPT}
 For $\xi=(-z_0, -i\omega, z_1)$ as in (\ref{daxi}), 
 let $\sigma_{\xi}=(Z_{\xi}, \aA_X)\in\vV_X$ be the associated 
 weak stability condition. 
 For an element $v=(-n, -\beta, 1)\in \Gamma$, 
 we have the following. 
 
(i) Assume that $\arg z_0<\arg z_1$. 
Then we have 
\begin{align*}
M^{v}(\sigma_{\xi})=\left\{\begin{array}{l}
\mbox{ ideal sheaves }I_C \subset \oO_X
\mbox{ for subschemes } \ 
C\subset X \\
\mbox{ with } \dim C\le 1, \ [C]=\beta \mbox{ and } \ \chi(\oO_C)=n.
\end{array}
\right\}.
\end{align*}

(ii) Assume that $\arg z_0>\arg z_1$. Then we have 
\begin{align*}
M^{v}(\sigma_{\xi})=\left\{\begin{array}{l}
\mbox{ two-term complexes }
(\oO_X \stackrel{s}{\to}F) \mbox{ for } \\
\mbox{ stable pairs }
(F, s)
\mbox{ with } [F]=\beta, \ \chi(F)=n.
\end{array}
\right\}.
\end{align*}
\end{prop}
Moreover in both cases, any $E\in M^v(\sigma_{\xi})$ is
$\sigma_{\xi}$-stable. 
\begin{proof}
(i)
Take $E\in M^{v}(\sigma_{\xi})$ and consider 
the exact sequence (\ref{useseq}). 
Suppose that $Q\neq 0$. 
Since $\arg z_0<\arg z_1$, we have 
$\arg Z_{\xi}(I_C)>\arg Z_{\xi}(Q[-1])$, which 
contradicts to the $\sigma_{\xi}$-semistability 
of $E$. Hence $Q=0$ and 
$E$ is isomorphic to the ideal sheaf $I_C$.
Conversely take an object
 $I_C \in \aA_X$ for a curve $C\subset X$, 
and
an exact sequence in $\aA_X$, 
\begin{align}\notag
0 \lr A \lr I_C \lr B \lr 0,
\end{align}
for non-zero $A, B\in \aA_X$. 
Since $\hH^1(I_C)=0$, we have $\hH^{1}(B)=0$, 
hence we have $\rk(B)=1$ and $\rk(A)=0$. 
This implies that $A=F[-1]$ for $F\in \Coh_{\le 1}(X)$, 
thus 
$\arg Z_{\xi}(A)<\arg Z_{\xi}(B)$
is satisfied
by $\arg z_1>\pi/2$. 
Therefore $I_C$ is $\sigma_{\xi}$-stable.  

(ii) Take $E\in M^{v}(\sigma_{\xi})$. 
Note that by Remark~\ref{Ox}, the object
$\oO_x[-1]\in \aA_X$ is $\sigma_{\xi}$-stable. 
The condition $\arg z_0>\arg z_1$ implies
$\arg Z_{\xi}(\oO_x[-1])>\arg Z_{\xi}(E)$, 
hence we have $\Hom(\oO_x[-1], E)=0$
for any closed point $x\in X$. 
By Proposition~\ref{firstshow} (iii), 
$E$ is 
isomorphic to $(\oO_X \stackrel{s}{\to} F)$ 
for a stable pair $(F, s)$. 
 Conversely take  
 $E=(\oO_X \stackrel{s}{\to} F)$
 for a stable pair $(F, s)$. 
 We take an exact sequence in $\aA_X$, 
 $$0 \lr A \lr E \lr B \lr 0, $$
 for $A, B\in \aA_X$. 
 Suppose that $\rk(B)=0$, hence $\rk(A)=1$. 
 Since there is a surjection of sheaves 
 $\hH^1(E) \twoheadrightarrow \hH^1(B)$ and 
 $\hH^1(E)$ is 0-dimensional, we have 
 $B=Q[-1]$ for a 0-dimensional sheaf
 $Q$. 
 Then $\arg Z_{\xi}(A)<\arg Z_{\xi}(B)$ is 
 satisfied by $\arg z_{0}>\arg z_{1}$. 
 If $\rk (B)>0$, then $\rk (B)=1$ and $\rk(A)=0$. 
 In this case, $A$ is written as $G[-1]$ 
 for $G\in \Coh_{\le 1}(X)$. 
 By Proposition~\ref{firstshow} (iii), 
 $G$ is not 0-dimensional, hence 
 $G$ is 1-dimensional. Then 
 $\arg Z_{\xi}(A)<\arg_{\xi}(B)$ is satisfied
 by $\arg z_1>\pi/2$, hence 
 $E$ is $\sigma_{\xi}$-stable. 
\end{proof}
 
 \subsection{Proof of Theorem~\ref{thm:main}}
 We show Theorem~\ref{thm:main},
 using the wall-crossing formula
 in a general setting.
 The following theorem is a summary of 
 the results in Section~\ref{sec:Wcro}.
 (We note that the filtration on $\Gamma$
 is not necessary given by (\ref{hook}) 
 in Theorem~\ref{general}.)
 \begin{thm}\label{general}\emph{\bf{[Corollary~\ref{cor:inde}]}}
 Let $\Gamma_{\bullet}$ be a filtration of
 $\Gamma=N_{\le 1}(X)\oplus \mathbb{Z}$ 
 satisfying
 \begin{align}\notag
\Gamma_{0} \subset \cdots \subset 
\Gamma_{N-1}=N_{\le 1}(X) \stackrel{i}{\hookrightarrow} 
\Gamma_{N}=\Gamma,
\end{align}
  via the inclusion $i(s, l)=(s, l, 0)$, and 
  $\vV \subset \Stab_{\Gamma_{\bullet}}(\dD_X)$ 
a connected subset satisfying Assumption~\ref{assum}
in Section~\ref{sec:walcro}. We have the following. 
\begin{itemize}
\item For $\sigma=(Z, \aA) \in \vV$ and
$v=(-n, -\beta, 1)\in \Gamma$, there 
is a counting invariant,
$$\widehat{\DT}_{n, \beta}(\sigma)\in \mathbb{Q}, $$
such that if 
the moduli stack of $\sigma$-semistable objects
$E\in \aA$ with $\cl(E)=v$, denoted by $\mM^v(\sigma)$, 
is written as $[M/\mathbb{G}_m]$ where $M$ is a scheme 
with $\mathbb{G}_m$ acting trivially, we have 
$\widehat{\DT}_{n, \beta}(\sigma)=\chi(M)$. 
\item 
Let $\widehat{\DT}(\sigma)$ and $\widehat{\DT}_{0}(\sigma)$
be the series, 
\begin{align*}
\widehat{\DT}(\sigma)&=
\sum_{n, \beta}\widehat{\DT}_{n, \beta}(\sigma)x^n y^{\beta}, \\
\widehat{\DT}_{0}(\sigma)&=
\sum_{(n, \beta)\in \Gamma_{0}}
\widehat{\DT}_{n, \beta}(\sigma)x^n y^{\beta}.
\end{align*}
Then the quotient series
$$\widehat{\DT}'(\sigma)\cneq \frac{\widehat{\DT}(\sigma)}{\widehat{\DT}_{0}(\sigma)},$$
is well-defined and does not depend on a general point 
$\sigma \in \vV$.  
(See Definition~\ref{defi:general} for general points.)
\end{itemize}
 \end{thm}
 Let $I_n(X, \beta)$ be the moduli space of 
 subschemes $C\subset X$ with $\dim C\le 1$ 
 and $[C]=\beta$, $\chi(\oO_C)=n$. 
Since $I_n(X, \beta)$ is a projective scheme, 
we can consider the generating series, 
 \begin{align*}
\widehat{\DT}(X)&=
\sum_{n, \beta}\chi(I_n(X, \beta))x^n y^{\beta}, \\
\widehat{\DT}_{0}(X)&=
\sum_{n}
\chi(I_n(X, 0))x^n.
\end{align*}
Let $P_n(X, \beta)$ be the moduli space of 
stable pairs $(F, s)$ with $[F]=\beta$, 
$\chi(F)=n$. 
In~\cite{PT}, it is proved that $P_n(X, \beta)$ is 
a fine projective moduli scheme. 
We consider the generating series, 
$$\widehat{\PT}(X)=\sum_{n, \beta}\chi(P_n(X, \beta))x^n y^{\beta}.$$
Applying Theorem~\ref{general}, we 
obtain the Euler characteristic version of 
DT/PT correspondence. 
\begin{thm}\label{main:dtpt}
We have the following equality of the generating series, 
$$\widehat{\DT}'(X)\cneq \frac{\widehat{\DT}(X)}{\widehat{\DT}_{0}(X)}
=\widehat{\PT}(X).$$
\end{thm}
 \begin{proof}
 By Lemma~\ref{check} below, we can apply Theorem~\ref{general}
 for $\vV_X\subset \Stab_{\Gamma_{\bullet}}(\dD_X)$
 given in Definition~\ref{def:vV}. 
 Take two elements (\ref{daxi}), 
 $$\xi=(-z_0, -i\omega, z_1), \quad \xi'=(-z_0', -i\omega, z_1'),$$
 such that $\arg z_0<\arg z_1$ and $\arg z_0'>\arg z_1'$. 
 By Proposition~\ref{firstshow}, we have 
 \begin{align}\label{mM^v}
 &\mM^v(\sigma_{\xi})=[I_n(X, \beta)/\mathbb{G}_m], \\
 \label{mM^v2}
 &\mM^v(\sigma_{\xi'})=[P_n(X, \beta)/\mathbb{G}_m],
 \end{align}
 where $\mathbb{G}_m$ acts on $I_n(X, \beta)$ and 
 $P_n(X, \beta)$ trivially. 
 Here we note that any $E\in M^v(\sigma_{\xi})$ or $E\in M^v(\sigma_{\xi'})$
 is stable by Proposition~\ref{firstshow}, hence
 $\Aut(E)=\mathbb{G}_m$. The 
 stabilizer groups $\mathbb{G}_m$
  in the stacks (\ref{mM^v}), (\ref{mM^v2}) are
  contributions of such trivial automorphisms. 
Applying Theorem~\ref{general}, we have 
 $$\widehat{\DT}'(X)=\widehat{\DT}'(\sigma_{\xi})=\widehat{\DT}'(\sigma_{\xi'})
 =\widehat{\PT}(X), $$
 as expected.   
 \end{proof}
 We have used the following lemma, which will 
 be proved in Section~\ref{sec:tech}. 
 \begin{lem}\label{check}
 The subset $\vV_X\subset \Stab_{\Gamma_{\bullet}}(\dD_X)$
 satisfies Assumption~\ref{assum} in Section~\ref{sec:walcro}. 
 \end{lem}
 
 \begin{rmk}
 It is also possible to construct 
 (usual) stability conditions on $\dD_X$. 
 For elements, 
 $$\alpha \in \mathbb{R}_{>0}, \ B+i\omega \in N^1(X)_{\mathbb{C}}, \
 \gamma \in \mathfrak{H}, $$
 with $\omega$ ample, we set
 $$Z\colon \Gamma \ni (s, l, r) \mapsto
 s\alpha -(B+i\omega)l+r\gamma \in \mathbb{C}.$$
 Then $(Z, \aA_X)$ satisfies (\ref{phase}), and 
 it determines an element of $\Stab_{\Gamma}(\dD_X)$
 if $\Imm \gamma>0$. 
 One can show that PT theory is 
 realized as stable objects with respect to
 such stability conditions.
 On the other hand, 
 DT theory does not appear as stable objects 
 with respect to the above 
 stability conditions. It might  
 exist stability conditions 
 in which DT-theory appears as 
 stable objects 
 after crossing the 
 wall $\Imm \gamma=0$. 
 However 
  we are unable to show
 the support property at the points $\Imm \gamma=0$,
 so the wall-crossing at such points cannot 
 be justified. 
 This
  is one of the reasons we work over 
 the space of weak stability conditions,
 rather than usual stability conditions.  
 \end{rmk}
 
\section{General framework}\label{sec:walcro}
\subsection{Moduli stacks}
In this paragraph, we give a framework 
to discuss the moduli problem of semistable 
objects in $\dD_X$. 
 Let us recall that 
there is an algebraic stack $\mM$
locally of finite type over $\mathbb{C}$, which
parameterizes $E\in D^b(\Coh(X))$ satisfying 
\begin{align}\label{Ext}
\Ext^i(E, E)=0, \quad \mbox{ for any }i<0.
\end{align}
(See~\cite{LIE}.)
Let $\mM_{0}$ be the fiber at $[0] \in \Pic(X)$
of the following morphism, 
$$\det \colon \mM \ni E \longmapsto \det E \in \Pic(X).$$
For any object $E\in \dD_X$, the corresponding 
$\mathbb{C}$-valued point $[E] \in \mM$ is 
contained in $\mM_0$.
Let $\aA \subset \dD_{X}$ be the heart of a bounded 
t-structure on $\dD_X$. 
We can consider the following (abstract) substack,
$$\oO bj(\aA) \subset \mM_{0},$$
which parameterizes objects $E\in \aA$.
The above stack decomposes as 
\begin{align*}
\oO bj(\aA)=\coprod _{v\in \Gamma} \oO bj^{v}(\aA),
\end{align*}
where $\oO bj^{v}(\aA)$
is the stack of 
objects $E\in \aA$ with $\cl(E)=v$. 

\subsection{Assumption}\label{subsection:assum}
Here we give a framework 
to discuss the wall-crossing formula under a general setting. 
Let $\Gamma_{\bullet}$ be a filtration (\ref{filt}) 
on $\Gamma=N_{\le 1}(X) \oplus \mathbb{Z}$, 
satisfying the following, 
\begin{align}\label{asfil}
\Gamma_{0} \subset \cdots \subset 
\Gamma_{N-1}=N_{\le 1}(X) \stackrel{i}{\hookrightarrow} 
\Gamma_{N}=\Gamma,
\end{align}
via the inclusion $i(s, l)=(s, l, 0)$. 
We assume that a connected subset 
$\vV\subset \Stab_{\Gamma_{\bullet}}(\dD_X)$ 
satisfies the following assumption. 
\begin{assum}\label{assum}
For any $\sigma=(Z=\{Z_i\}_{i=0}^{N}, \pP) \in \vV$ with 
$\aA=\pP((0, 1])$, the following conditions 
are satisfied. 
\begin{itemize}
\item There is $\psi \in \mathbb{R}$ which satisfies  
\begin{align}\label{phase1}
\oO_X \in \pP(\psi), \quad 
\frac{1}{2}<\psi < 1, 
\end{align}
and $\oO_X$ is the only object $E\in \pP(\psi)$
with $\cl(E)=(0, 0, 1)$. 
\item 
For any $1\le j\le N-1$, we have 
\begin{align}
Z_j(\mathbb{H}_j) \subset \mathbb{R}\cdot i.
\end{align}
\item For any $v, v'\in \Gamma_0$ and 
any other point $\tau=(W, \qQ) \in \vV$, 
 we have 
\begin{align}\label{extra}
Z(v)\in \mathbb{R}_{>0}Z(v') \quad 
\mbox{ if and only if } \quad 
W(v) \in \mathbb{R}_{>0}W(v').
\end{align}
\item For any $v\in \Gamma$ with $\rk(v)= 1$
or $v\in \Gamma_0$, the stack 
of objects 
$$\oO bj^{v}(\aA) \subset \mM_0,$$
is an open substack of $\mM_0$. In particular, 
$\oO bj^{v}(\aA)$ is an algebraic stack 
locally of finite type over $\mathbb{C}$. 
\item For any $v\in \Gamma$ with $\rk(v)=1$ or 
$v\in \Gamma_{0}$, the stack of 
$\sigma$-semistable objects $E\in \aA$
with $\cl(E)=v$, 
$$\mM^v(\sigma)\subset \oO bj^{v}(\aA), $$
is an open substack of finite type 
over $\mathbb{C}$. 
\item There 
are subsets $0\in T\subset S\subset N_{\le 1}(X)$, 
which satisfy Assumption~\ref{assum2} 
in the next paragraph.  
\item For any other point $\tau \in \vV$, 
there is a good path 
(see Definition~\ref{def:good} below) in $\vV$ which connects 
$\sigma$ and $\tau$. 
\end{itemize}
\end{assum}
The notion of good path is defined as follows. 
\begin{defi}\label{def:good}
\emph{
A path $[0, 1] \ni t \mapsto \sigma_t \in \vV$
is \textit{good} 
if for any $t\in (0, 1)$ and 
$v\in \Gamma_{0}$ satisfying 
$Z_t(v) \in \mathbb{R}_{>0}Z_t(\oO_X)$, we have 
\begin{align}
\label{good2}
&\arg Z_{t+\varepsilon}(v)< \arg Z_{t+\varepsilon}(\oO_X), \quad 
\arg Z_{t-\varepsilon}(v)> \arg Z_{t-\varepsilon}(\oO_X), \ \mbox{or} \\
\label{good1}
&\arg Z_{t+\varepsilon}(v)> \arg Z_{t+\varepsilon}(\oO_X), \quad 
\arg Z_{t-\varepsilon}(v)< \arg Z_{t-\varepsilon}(\oO_X), 
\end{align}
for $0<\varepsilon \ll 1$. }
\end{defi}
\begin{rmk}\label{rk}
For $\sigma=(Z, \pP)\in \vV$, 
the first condition of Assumption~\ref{assum}
implies that $\rk(E)\ge 0$ for any $E\in \pP((0, 1])$. 
\end{rmk}

\subsection{Completions of 
$\mathbb{C}[N_{\le 1}(X)]$}
Here we discuss about
completions of $\mathbb{C}[N_{\le 1}(X)]$
corresponding to subsets
 $0\in T\subset S\subset \mathbb{C}[N_{\le 1}(X)]$
 satisfying Assumption~\ref{assum2} below. 
 Note that the existence of such $T$, $S$
 is one of the conditions of Assumption~\ref{assum}.
 For subsets $S_1, S_2 \subset N_{\le 1}(X)$,  
we set
$$S_1+S_2 \cneq \{s_1+s_2 : s_i \in S_i \}\subset
N_{\le 1}(X).$$
The sixth condition of Assumption~\ref{assum} is stated 
as follows. 
\begin{assum}\label{assum2}
In the situation of Assumption~\ref{assum}, 
the subsets $0\in T\subset S\subset N_{\le 1}(X)$
satisfy the following conditions. 
\begin{itemize}
\item We have 
\begin{align}\label{assum:pro}
T+T\subset T, \quad S+T \subset S.
\end{align}
\item For any $x\in N_{\le 1}(X)$, 
there are only finitely many ways to write 
$x=y+z$ for $y, z\in S$. 
\item 
Let $\psi \in \mathbb{R}$ be as in (\ref{phase1})
for $\sigma \in \vV$.  
Then for $I=(\psi-\varepsilon, \psi+\varepsilon)$
with $0<\varepsilon \ll 1$, we have 
(see~(\ref{CI}) for $C_{\sigma}(I)$)
\begin{align}
\label{S1}
&\{(n, \beta)\in N_{\le 1}(X) : 
(-n, -\beta, 1)\in C_{\sigma}(I)\}
\subset S, \\
\label{T1}
&\{(n, \beta)\in \Gamma_{0} : 
(-n, -\beta, r)\in C_{\sigma}(I), \ 
r=0 \mbox{ or } 1\}
\subset T.
\end{align}
\item There is a family of sets
$\{S_{\lambda}\}_{\lambda \in \Lambda}$
with $S_{\lambda}\subset S$ such that 
$S\setminus S_{\lambda}$ is a finite set and 
\begin{align*}
S_{\lambda}+T\subset S_{\lambda}, \quad
 S=\bigcup_{\lambda\in \Lambda}(S\setminus S_{\lambda}).
\end{align*}
\end{itemize}
 \end{assum}
 For a possibly infinite sum, 
 $$f=\sum_{n, \beta}a_{n, \beta}x^n y^{\beta}, \quad 
 a_{n, \beta} \in \mathbb{C}, $$
 its support is defined by 
 $$\Supp(f)\cneq \{ (n, \beta)\in N_{\le 1}(X) :
 a_{n, \beta}\neq 0\}.$$
 The completions
are defined as follows. 
\begin{defi}\emph{
For a subset $S\subset N_{\le 1}(X)$, 
the vector space $\kakkoS$ 
is defined by 
$$\kakkoS\cneq \left\{ f=\sum_{n, \beta}a_{n, \beta}x^n y^{\beta}
: \Supp(f)\subset S \right\}.$$}
\end{defi}
Suppose that $0\in T\subset S$ satisfy Assumption~\ref{assum2}. 
The product on 
$\mathbb{C}[N_{\le 1}(X)]$ generalizes naturally to 
products on $\kakkoT$, 
and $\kakkoS$ is a $\kakkoT$-module
with $\kakkoT\subset \kakkoS$. 
Let $\{f_i\}_{i\in I}$ be 
a possibly infinite family
of elements of $\kakkoT$ with $f_i(0, 0)=0$. 
Then the infinite product 
 of the exponential
 makes sense,  
 $$\prod_{i\in I}\exp(f_i)\in \kakkoT, $$
 if the following condition holds
 for any $(n, \beta) \in N_{\le 1}(X)$,  
 \begin{align}\label{sharp}
 \sharp \{ i\in I : (n, \beta)\in \Supp(f_i)\}<\infty.
 \end{align}
  Also for $f\in \kakkoS$ and $g\in \kakkoT$ with 
  $g(0, 0)\neq 0$, 
  the quotient series makes sense, 
$$\frac{f}{g} \in \kakkoS.$$
Let $\{S_{\lambda}\}_{\lambda \in \Lambda}$ be
as in Assumption~\ref{assum2}. 
We write $\lambda' \preceq \lambda$ if 
$S_{\lambda}\subset S_{\lambda'}$. 
For $\lambda'\preceq \lambda$, we have
the surjection of $\kakkoT$-modules, 
$$\kakkoS /\kakkoSl \twoheadrightarrow 
\kakkoS/\kakkoSll.$$
In this way, we obtain 
the inductive system
of $\kakkoT$-modules
 $\{\kakkoS /\kakkoSl\}_{\lambda \in \Lambda}$, and
the isomorphism of $\kakkoT$-modules,
\begin{align}\label{prolim}
\kakkoS
\cong
\lim_{\begin{subarray}{c}\longleftarrow \\
\lambda \in \Lambda
\end{subarray}}
\kakkoS/
\kakkoSl.
\end{align}
Since each $\kakkoS/
\kakkoSl$
is a finite dimensional $\mathbb{C}$-vector space, 
its Euclid topology together with the isomorphism (\ref{prolim})
induce the topology on $\kakkoS$.

\subsection{Joyce invariants}\label{subsec:Joyce}
Take $\sigma \in \vV$, $v\in \Gamma$ with 
$\rk(v)=1$ or $v\in \Gamma_0$. 
Under Assumption~\ref{assum},  
we are able to construct the $\mathbb{Q}$-valued 
invariant, 
\begin{align}\notag
J^v(\sigma) \in \mathbb{Q}, 
\end{align}
such that if $\mM^v(\sigma)$ is written as
$[M/\mathbb{G}_m]$ for a scheme $M$
with $\mathbb{G}_m$ acting on $M$ trivially, 
then 
\begin{align}\label{Jeuler}
J^v(\sigma)=\chi(M).
\end{align}
Here $\chi(\ast)$ is the topological Euler characteristic. 
In general $\mM^v(\sigma)$ includes information of
the automorphisms of strictly semistable objects, and 
the denominator of $J^v(\sigma)$ 
is contributed by such non-trivial automorphisms. 
The invariant $J^v(\sigma)$ is introduced by 
D.~Joyce~\cite{Joy4}, using the 
notion of Hall-algebras. Here we briefly explain 
how to construct $J^v(\sigma)$. 

Suppose for instance that $\aA \subset \dD_{X}$ is the 
heart of bounded t-structure on $\dD_X$, 
such that the stack $\oO bj(\aA)$ is 
an algebraic stack locally of finite type. 
We denote by $\eE x(\aA)$ the stack of 
short exact sequences in $\aA$. There are morphisms 
of stacks, 
$$p_i \colon \eE x(\aA) \lr \oO bj(\aA), $$
sending a short exact sequence 
$$0 \lr A_1 \lr A_2 \lr A_3 \lr 0$$
to objects $A_i$ respectively. 

The $\mathbb{C}$-vector space $\hH(\aA)$ is defined 
to be spanned by symbols, 
$$[\xX \stackrel{f}{\lr} \oO bj(\aA)],$$
where $\xX$ is an algebraic stack of finite type 
with affine stabilizers, and $f$ is a morphism of 
stacks. The relations are generated of the form, 
\begin{align}\label{mot}
[\xX \stackrel{f}{\lr}\oO bj(\aA)]
-[\yY \stackrel{f|_{\yY}}{\lr} \oO bj(\aA)]-
[\uU \stackrel{f|_{\uU}}{\lr}\oO bj(\aA)], 
\end{align}
for a closed substack $\yY \subset \xX$ and 
$\uU=\xX\setminus \yY$. 

There is an associative product on $\hH(\aA)$
based on Ringel-Hall algebras, defined by 
$$[\xX \stackrel{f}{\lr}\oO bj(\aA)] \ast 
[\yY \stackrel{g}{\lr}\oO bj(\aA)]
=[\zZ \stackrel{p_2 \circ h}{\lr}\oO bj(\aA)],$$
where the morphism $h$ fits into the Cartesian square
$$\xymatrix{
\zZ \ar[r]^{h} \ar[d] & \eE x(\aA) \ar[d]^{(p_1, p_3)}
 \ar[r]^{p_2} & \oO bj(\aA). \\
\xX \times \yY  \ar[r]^{f\times g} & \oO bj(\aA)^{\times 2} &
}$$
The $\ast$-product is 
associative by~\cite[Theorem~5.2]{Joy2}. 
The algebra $\hH(\aA)$ is $\Gamma$-graded,
$$\hH(\aA)=\bigoplus_{v\in \Gamma}\hH_{v}(\aA),$$
where $\hH_v(\aA)$ is spanned by $[\xX \stackrel{f}{\to} \oO bj(\aA)]$
factoring via $\oO bj^{v}(\aA)\subset \oO bj(\aA)$. 

Let $\vV \subset \Stab_{\Gamma_{\bullet}}(\dD_X)$
be a subset satisfying Assumption~\ref{assum}, 
and take $\sigma=(Z, \pP) \in \vV$ 
with $\aA=\pP((0, 1])$. 
In Assumption~\ref{assum}, we do
not assume 
that $\oO bj^{v}(\aA)$ is algebraic 
for $\rk(v)>1$ or $\rk(v)=0$, $v\notin \Gamma_{0}$. 
 Instead we can discuss as follows. 
Under Assumption~\ref{assum}, we are able to 
define the following vector spaces, 
\begin{align*}\hH_{0}(\aA)&=\bigoplus_{v\in \Gamma_{0}}
\hH_{v}(\aA), \\
\hH_{N}(\aA)&=\bigoplus_{\rk(v)=1}
\hH_{v}(\aA).
\end{align*}
The similar $\ast$-product makes
$\hH_{0}(\aA)$ an associate algebra, 
and $\hH_{N}(\aA)$ a $\hH_{0}(\aA)$-bimodule. 
We define the elements $\delta^{v}(\sigma)$
and $\epsilon^{v}(\sigma)$ as follows. 
\begin{defi}\label{edel}
\emph{
Under the above situation, 
take $v\in \Gamma$ with $\rk(v)=1$
or $v\in \Gamma_{0}$. Suppose that 
$v\in C_{\sigma}(\phi)$ with $0<\phi \le 1$. 
We define $\delta^{v}(\sigma)$ to be 
\begin{align*}
\delta^{v}(\sigma)&=[\mM^v(\sigma)\hookrightarrow \oO bj(\aA)]
\in \hH_{\ast}(\aA),
\end{align*}
where $\ast=0$ if $v\in \Gamma_0$ and $\ast=N$
if $\rk(v)=1$. 
The element $\epsilon^{v}(\sigma)\in \hH_{\ast}(\aA)$
is defined to be 
\begin{align}\label{ep}
\epsilon^{v}(\sigma)=\sum_{\begin{subarray}{c}
v_1+\cdots +v_l=v, \\
v_i \in C_{\sigma}(\phi), \ 1\le i\le l.
\end{subarray}}
\frac{(-1)^{l-1}}{l}\delta^{v_1}(\sigma)\ast \cdots 
\ast \delta^{v_l}(\sigma). 
\end{align}}
\end{defi}
\begin{rmk}
It is possible to 
define $\delta^{v}(\sigma)$ by 
the fourth condition of Assumption~\ref{assum}. 
Take $v_1, \cdots, v_l \in C_{\sigma}(\phi)$ 
which appear in (\ref{ep}). 
By Remark~\ref{rk} and Lemma~\ref{Sfin} below, 
there is $1\le e \le l$ such that 
$\rk(v_e)=1$ and $v_i \in \Gamma_{0}$ for 
$i\neq e$, and there is a finite number of 
possibilities for such $v_i$.  
Therefore 
(\ref{ep}) is a finite sum and 
$\epsilon^{v}(\sigma)$ is well-defined. 
\end{rmk}
There is a map (cf.~\cite[Theorem~4.9]{Joy5}), 
$$\Upsilon \colon \hH(\aA) \lr \mathbb{Q}(t), $$
such that if $G$ is a special
algebraic 
group (cf.~\cite[Definition~2.1]{Joy5}) 
acting on a variety $Y$, we have 
$$\Upsilon([[Y/G] \stackrel{f}{\to}\oO bj(\aA)])
=P(Y, t)/P(G, t), $$
where $P(Y, t)$ is the virtual Poincar\'e polynomial of 
$Y$, i.e. if $Y$ is smooth and projective, we have 
$$P(Y, t)=\sum_{i\in \mathbb{Z}}(-1)^{i} \dim H^i(Y, \mathbb{C})t^i, $$
and $P(Y, t)$ is defined for any variety $Y$ using the motivic
relation (\ref{mot}) for varieties. 
\begin{thm}\emph{(\cite[Section~6.2]{Joy5})}
The element 
$$(t^2-1)\Upsilon(\epsilon^{v}(\sigma)) \in \mathbb{Q}(t),$$
is regular at $t=1$. 
\end{thm}
The above theorem is used to define the invariant 
$J^v(\sigma)\in \mathbb{Q}$. 
\begin{defi}\label{Jinv}
\emph{
For $\sigma \in \vV$ and $v\in \Gamma$ with 
$\rk(v)=1$ or $v\in \Gamma_0$, 
we define $J^v(\sigma)\in \mathbb{Q}$ 
as follows. }
\begin{itemize}
\item \emph{If $v\in C_{\sigma}(\phi)$ for $0<\phi \le 1$, we define}
\begin{align*}
J^v(\sigma)\cneq
\lim_{t\to 1}
(t^2-1)\Upsilon(\epsilon^{v}(\sigma)).
\end{align*}
\item \emph{If $v\in C_{\sigma}(\phi)$ for $1<\phi \le 2$, we 
define $J^v(\sigma)\cneq J^{-v}(\sigma)$.}
\item \emph{Otherwise we define $J^v(\sigma)=0$. }
\end{itemize}
\end{defi}
\begin{rmk}
Suppose that $\mM^v(\sigma)=[M/\mathbb{G}_m]$ for a scheme $M$
with $\mathbb{G}_m$ acting on $M$ trivially. 
Then for any $\mathbb{C}$-valued point 
of $\mM^v(\sigma)$, 
the corresponding object
$E\in \aA$ is $\sigma$-stable. Hence we have 
$\epsilon^{v}(\sigma)=\delta^{v}(\sigma)=
([M/\mathbb{G}_m] \to \oO bj(\aA))$ and 
$$(t^2-1)\Upsilon(\epsilon^{v}(\sigma))=P(M, t).$$
Therefore we obtain (\ref{Jeuler}). 
\end{rmk}
Under the above situation, we introduce 
the following generating series. 
\begin{defi}\label{def:geDT}
\emph{
Let $\vV\subset \Stab_{\Gamma_{\bullet}}(\dD_X)$
be a subset satisfying Assumption~\ref{assum}. 
For $(n, \beta)\in N_{\le 1}(X)$
and $\sigma \in \vV$, 
we define $\widehat{\DT}_{n, \beta}(\sigma)$ to be
$$\widehat{\DT}_{n, \beta}(\sigma)\cneq 
J^{(-n, -\beta, 1)}(\sigma) \in \mathbb{Q}.$$
The generating series $\widehat{\DT}(\sigma)$
and $\widehat{\DT}_{0}(\sigma)$ are defined by 
\begin{align*}
\widehat{\DT}(\sigma)&\cneq
\sum_{n, \beta}\widehat{\DT}_{n, \beta}(\sigma)x^n y^{\beta}\in 
\kakkoS, \\
\widehat{\DT}_{0}(\sigma)&\cneq 
\sum_{(n, \beta)\in \Gamma_{0}}
\widehat{\DT}_{n, \beta}(\sigma)x^n y^{\beta}\in 
\kakkoT.
\end{align*}
The above series are elements of $\kakkoS$, $\kakkoT$
respectively by the third condition of Assumption~\ref{assum2}. 
Since $\widehat{\DT}_0(\sigma)=1+\cdots$ by the first 
condition of Assumption~\ref{assum}, 
the following reduced series is well-defined, 
\begin{align*}
\widehat{\DT}'(\sigma)\cneq \frac{\widehat{\DT}(\sigma)}{\widehat{\DT}_{0}(\sigma)}
\in \kakkoS.
\end{align*}}
\end{defi}

\section{Wall-crossing formula}\label{sec:Wcro}
The purpose of this section is 
to study how $\widehat{\DT}(\sigma)$ varies 
under change of $\sigma$. 
First let us introduce the 
pairing $\chi$ on $\Gamma$, 
\begin{align}\label{chi}
\chi((s, l, r), (s', l', r'))=rs'-r's.
\end{align}
By the Riemann-Roch theorem and the Serre duality, for 
$E, F \in \dD$ we have 
\begin{align*}\chi(\cl(E), \cl(F))=&
\dim \Hom(E, F)-\dim \Ext^1(E, F) \\
& \qquad \qquad +\dim \Ext^1(F, E)
-\dim \Hom(F, E).
\end{align*}
Below we fix a subset $\vV \subset \Stab_{\Gamma_{\bullet}}(\dD_X)$
satisfying Assumption~\ref{assum}. 
\subsection{Wall and chamber structure}
In this paragraph, 
we show the existence of wall and 
chamber structure in a neighborhood of 
$\sigma \in \vV$. 
For
$v\in \Gamma$ and $\varepsilon>0$, we set 
\begin{align}\label{S_e}
S_{\varepsilon, v}(\sigma)\cneq
\left\{v'\in \Gamma : \begin{array}{l}
\mbox{there is }\phi \in \mathbb{R} \mbox{ such that }\\
v', v-v' \in C_{\sigma}((\phi-\varepsilon, \phi+\varepsilon)).
\end{array}
\right\}.
\end{align}
We show the following lemma. 
\begin{lem}\label{Sfin}
(i) Suppose that $v\in \Gamma_0$.  Then for 
$0<\varepsilon\ll 1$, we have 
$S_{\varepsilon, v}(\sigma)\subset \Gamma_0$ and 
$S_{\varepsilon, v}(\sigma)$ is a finite set. 

(ii) Suppose that $\rk(v)=1$. Then for 
$0<\varepsilon \ll 1$, we have 
$S_{\varepsilon, v}(\sigma)\cap N_{\le 1}(X)\subset
\Gamma_{0}$ and $S_{\varepsilon, v}(\sigma)$ is a finite set. 
\end{lem}
\begin{proof}
(i) The first assertion is obvious. 
The finiteness of 
$S_{\varepsilon, v}(\sigma)$ easily follows 
from the support property (\ref{support}). 

(ii) 
Suppose that $S_{\varepsilon, v}(\sigma)\neq \emptyset$. 
Then we have $v\in C_{\sigma}((\phi-\varepsilon, \phi+\varepsilon))$
for some $0\le \phi<2$. Since $\phi_{\sigma}(v)=\psi$, 
where $\psi$ is given in~(\ref{phase1}), 
we have $\psi\in (\phi-\varepsilon, \phi+\varepsilon)$, 
hence $v\in C_{\sigma}((\psi-2\varepsilon, \psi+2\varepsilon))$. 
By choosing $\varepsilon>0$ sufficiently small, 
we may assume that $(\psi-2\varepsilon, \psi+2\varepsilon)
\subset (1/2, 1)$. 
Then if $v=v'+v''$ in  
$C_{\sigma}((\psi-2\varepsilon, \psi+2\varepsilon))$, 
we may assume that $v'$ and $v''$ 
are written as $v'=(-n', -\beta', 1)$
and $v''=(-n'', -\beta'', 0)$
by Remark~\ref{rk}. 
Then $v''\in \Gamma_0$ follows from the second 
condition of Assumption~\ref{assum}, 
which implies the first assertion. 
By (\ref{S1}) and (\ref{T1}), 
we have $(n', \beta')\in S$ and $(n'', \beta'') \in T$. 
Therefore the finiteness of $S_{\varepsilon, v}(\sigma)$
follows from the second condition of Assumption~\ref{assum2}. 
\end{proof}
We have the following lemma. 

\begin{lem}\label{wacham}
Take $\sigma \in \vV$, 
$v\in \Gamma$ with $\rk(v)=1$ or $v\in \Gamma_{0}$, 
and $0<\varepsilon \ll 1$
such that $S_{\varepsilon, v}(\sigma)$ is a finite 
set. (cf.~Lemma~\ref{Sfin}.)
Let $\sigma \in U_{\varepsilon}
 \subset \Stab_{\Gamma_{\bullet}}(\dD_X)$
 be an open neighborhood of $\sigma$ such that any
 $\tau=(W, \qQ)\in U_{\varepsilon}$
  satisfies $d(\pP, \qQ)<\varepsilon$. 
Then 
there are finitely many real codimension 
one submanifolds $\{\wW_{\lambda}\}_{\lambda \in \Lambda}$
in $U_{\varepsilon}$,
such that if $\sigma_1, \sigma_2\in \vV$ are 
contained in the same connected component of 
$U_{\varepsilon}
\setminus\{\wW_{\lambda}\}_{\lambda \in \Lambda}$, then 
$$\mM^{v'}(\sigma_1)=\mM^{v'}(\sigma_2),$$
for any $v'\in S_{\varepsilon, v}(\sigma)$. 
\end{lem}
\begin{proof}
We set $\Lambda=S_{\varepsilon, v}(\sigma)\times 
S_{\varepsilon, v}(\sigma)$. For each $\gamma=(v_1, v_2)
\in \Lambda$, 
we define $\wW_{\lambda}$ to be
$$\wW_{\lambda}\cneq \{\tau=(W, \qQ) \in U_{\varepsilon} :
W(v_1)/W(v_2) \in \mathbb{R}_{>0}\}.$$
Then it is easy to see that
$\{\wW_{\lambda}\}_{\gamma\in \Lambda}$ gives 
a desired set of submanifolds. 
(Also see the proof of~\cite[Proposition~9.3]{Brs2}.)
\end{proof}

\subsection{Joyce's formula}
Take $\sigma\in \vV$ and
$v\in \Gamma$ with $\rk(v)=1$ or $v\in \Gamma_0$. 
Our setting in this paragraph is as follows. 
\begin{itemize}
\item 
We choose
$\varepsilon>0$ and an open
neighborhood $\sigma \in U_{\varepsilon}$ as in Lemma~\ref{wacham}. 
By choosing $\varepsilon>0$ sufficiently small, 
we may assume that any connected 
component $\cC \subset U_{\varepsilon}
\setminus \{\wW_{\lambda}\}_{\lambda \in \Lambda}$
satisfies $\sigma \in \overline{\cC}$.
Below we denote by $V_{\varepsilon}$ the connected component of 
$U_{\varepsilon}\cap \vV$ which contains $\sigma$. 
We take two weak stability conditions 
$\tau, \tau'\in V_{\varepsilon}$. 
\end{itemize}
The wall-crossing formula enables us to 
describe $J^v(\tau)$ in
terms of $J^{v'}(\tau)$ with 
$v'\in S_{\varepsilon, v}(\sigma)$. 
The transformation coefficients are purely 
combinatorial.
In what follows, $I\subset \mathbb{R}$ is 
a sufficiently small interval, i.e. 
$I=(a, b)$ with $0<b-a \ll 1$. 
Note that for $v\in C_{\sigma}(I)$, 
we have, (see~(\ref{def:phase}),)
$$\phi_{\tau}(v) \in I\pm \varepsilon.$$
\begin{defi}\label{def:S}
\emph{\bf{\cite[Definition~4.2]{Joy4}}}\emph{
For non-zero $v_1, \cdots, v_l \in C_{\sigma}(I)$, 
we define 
$$S(\{v_1, \cdots, v_l\}, \tau, \tau')
\in \{0, \pm 1\}, $$
as follows. 
If for each $i=1, \cdots, l-1$, we have either 
(\ref{eith1}) or (\ref{eigh2}), 
\begin{align}
\label{eith1}
\phi_{\tau}(v_i) \le \phi_{\tau}(v_{i+1}) 
& \mbox{ and } \phi_{\tau'}(v_1 +\cdots + v_i) 
> \phi_{\tau'}(v_{i+1}+
\cdots +v_l), \\
\label{eigh2}
\phi_{\tau}(v_i)> \phi_{\tau}(v_{i+1}) 
& \mbox{ and } \phi_{\tau'}(v_1+\cdots +v_i) \le
 \phi_{\tau'}(v_{i+1}+
\cdots +v_l),
\end{align}
then define 
$S(\{v_1, \cdots, v_l\}, \tau, \tau')$ to be $(-1)^r$, where 
$r$ is the number of $i=1, \cdots, l-1$ satisfying (\ref{eith1}). 
Otherwise we define 
$S(\{v_1, \cdots, v_l\}, \tau, \tau')=0$.} 
\end{defi}
Another combinatorial coefficient is 
defined as follows. 
\begin{defi}\emph{\bf{\cite[Definition~4.4]{Joy4}}}\label{defi:U}
\emph{
For non-zero $v_1, \cdots, v_l \in C_{\sigma}(I)$,
 we
define 
\begin{align}\label{def:U}
&U(\{v_1, \cdots, v_l\}, \tau, \tau')=  
\sum_{1\le l'' \le l' \le l}
\sum_{\begin{subarray}{c}
\psi \colon \{1, \cdots, l\} \to \{1, \cdots, l'\},
\\
 \xi \colon \{1, \cdots, l'\} \to \{1, \cdots, l''\}.
\end{subarray}}\\
&\qquad \qquad \qquad \prod_{a=1}^{l''}
S(\{ w_i \}_{i\in \xi^{-1}(a)}, \tau, \tau') 
 \frac{(-1)^{l''}}{l''}
\prod_{b=1}^{l'}\frac{1}{\lvert \psi^{-1}(b)\rvert!}. \label{U}
\end{align}
Here $\psi$, $\xi$
satisfy the following. }
\begin{itemize}
\item \emph{$\psi$ and $\xi$ are non-decreasing surjective maps. }
\item \emph{  
We have 
\begin{align}\label{tau0}
\phi_{\tau}(v_i)=\phi_{\tau}(v_j),
\end{align}
for $1\le i, j\le l$ with $\psi(i)=\psi(j)$. }
\item \emph{For $1\le i, j\le l''$, we have} 
\begin{align}\label{tau}
\phi_{\tau'}(\sum_{k\in \psi^{-1}\xi^{-1}(i)}v_k)= 
\phi_{\tau'}(\sum_{k\in\psi^{-1}\xi^{-1}(j)}v_{k}).
\end{align}
\end{itemize}
\emph{Also $w_i$ for $1\le i\le l'$ is defined as} 
\begin{align}\label{wi}
w_i=\sum_{j\in \psi^{-1}(i)}v_j\in
C_{\sigma}(I).
\end{align}
\end{defi}
Note that if $S_{\varepsilon, v}(\sigma)\neq \{0\}$, 
then $v\in C_{\sigma}((\phi-\varepsilon, \phi+\varepsilon))$
for some $\phi \in \mathbb{R}$ and 
$S_{\varepsilon, v}(\sigma)\subset C_{\sigma}((\phi-\varepsilon, 
\phi+\varepsilon))$. By choosing $\varepsilon>0$ sufficiently small, 
we can take $I$ to be $I=(\phi-\varepsilon, \phi+\varepsilon)$
in Definition~\ref{def:S}
and Definition~\ref{defi:U}. 
In our situation, Joyce's
wall-crossing formula~\cite{Joy4} is applied to 
show the following. 
(See the discussion in~\cite[Section~2]{Tolim2}.) 
\begin{thm}\emph{\bf{\cite[Theorem~6.28, Equation (130)]{Joy4}}}
\label{prop:trans}
We have 
\begin{align}
\notag
J^v(\tau')=&\sum_{\begin{subarray}{c}l\ge 1, \ v_i \in 
S_{\varepsilon, v}(\sigma), \\
v_1+\cdots +v_l=v\end{subarray}}
\sum _{\begin{subarray}{c}
G \text{ \rm{is a connected, simply connected oriented}} \\
\text{\rm{graph with vertex} }\{1, \cdots, l\}, \
\stackrel{i}{\bullet}\to \stackrel{j}{\bullet}\text{ \rm{implies} }
i< j
\end{subarray}} \\
\label{Trans2}
&\qquad \frac{1}{2^{l-1}}U(\{v_1, \cdots, v_l\}, \tau, \tau')
\prod_{\stackrel{i}{\bullet} \to \stackrel{j}{\bullet}\text{ \rm{in} }G}
\chi(v_i, v_j)\prod_{i=1}^{l}J^{v_i}(\tau).
\end{align} 
\end{thm}
\begin{rmk}
The property that the set of $\sigma \in \vV$ in which a
fixed $E\in \dD_X$ is semistable 
is closed (cf.~Remark~\ref{rmk:closed}) 
corresponds to the dominant condition in the sense 
of~\cite[Definition~3.16]{Joy4}. 
This property is not true for another 
generalized stability conditions 
whose central charges are polynomials~\cite{Bay}, \cite{Tolim}. 
This is one of the reasons we work over weak 
stability conditions rather than polynomial 
stability conditions. 
\end{rmk}
The above theorem immediately yields the 
following. 
\begin{prop-defi}\label{strict}
\emph{
For $v=(-n, -\beta, 0) \in \Gamma_0$, 
the value $J^v(\tau)$ does not depend on $\tau\in \vV$. 
We define} 
\begin{align}\label{defN}
\widehat{N}_{n, \beta}\cneq J^{v}(\tau),
\end{align}
\emph{for $\tau \in \vV$.} 
\end{prop-defi}
\begin{proof}
Since $\vV$ is connected by our assumption, the 
problem is local on $\vV$. 
Noting Lemma~\ref{Sfin} (i) and $\chi(\ast, \ast)=0$ 
on $\Gamma_0$, the formula 
(\ref{Trans2}) implies 
$J^v(\tau)=J^v(\tau')$
for any $\tau, \tau'\in V_{\varepsilon}$. 
\end{proof}
\subsection{Wall-crossing formula of  
generating functions}
Take $\sigma=(Z, \pP) \in \vV$ and a 
continuous family in $\vV$, 
$$(-\delta, \delta) \ni t \mapsto \sigma_t \in \vV,$$
with $\sigma_0=\sigma$
and $\delta >0$. 
By Lemma~\ref{wacham},
the following limiting series makes sense, 
\begin{align*}\widehat{\DT}(\sigma_{\pm}) &\cneq 
\lim_{t\to \pm 0}\widehat{\DT}(\sigma_{t}), \\
&=\sum_{n, \beta}\widehat{\DT}_{n, \beta}(\sigma_{\pm})x^n y^{\beta} \in 
\kakkoS,
\end{align*}
where $S$ is given in Assumption~\ref{assum2}. 
The series 
$$\widehat{\DT}_0(\sigma_{\pm}) \cneq 
\lim_{t\to \pm 0}\widehat{\DT}_0(\sigma_{t})
\in \kakkoT,$$
is also defined.
We set $W$ to be 
$$W\cneq \{v\in \Gamma_{0} :
Z(v)\in \mathbb{R}_{>0} Z(\oO_X)\}. $$
For any $v\in W$, we assume that 
\begin{align}\label{ass1}
&\arg Z_t(v)>\arg Z_{t}(\oO_X), \quad 0<t \ll 1, \\
\label{ass2}
&\arg Z_t(v)<\arg Z_{t}(\oO_X), \quad 0<-t \ll 1,
\end{align}
i.e. 
(\ref{good1}) happens 
at $t=0$. 
The following theorem is a generalization of
the result in~\cite{Tolim2}. 
\begin{thm}\label{main:DT}
We have the following equalities of the generating series, 
\begin{align}\label{thm:gen1}
\widehat{\DT}(\sigma_{-})&=
\widehat{\DT}(\sigma_{+})\cdot\prod_{-(n, \beta)\in W}
\exp(n\widehat{N}_{n, \beta}x^n y^{\beta}), \\
\label{thm:gen2}
\widehat{\DT}_{0}(\sigma_{-})&=
\widehat{\DT}_{0}(\sigma_{+})\cdot\prod_{-(n, \beta)\in W}
\exp(n\widehat{N}_{n, \beta}x^n y^{\beta}). 
\end{align}
\end{thm}
\begin{proof}
We only show (\ref{thm:gen1}), as
(\ref{thm:gen2}) is similarly proved.  
The proof goes along with the same
argument of~\cite[Theorem~4.7]{Tolim2}. 
Take $v\in \Gamma$ with $\rk(v)=1$  
and $\varepsilon>0$ so that 
$S_{\varepsilon, v}(\sigma)$ is a finite set. 
For elements 
$v', v''\in S_{\varepsilon, v}(\sigma)$, we write 
$$\phi_{\pm}(v') \le \phi_{\pm}(v''),$$
if 
$\phi_{\sigma_t}(v') \le \phi_{\sigma_t}(v'')$
holds for $0<\pm t \ll 1$. 
 Note that for $v'\in S_{\varepsilon, v}(\sigma)$, 
 we have 
 \begin{align}\label{iff}
 \phi_{\pm}(v')=\phi_{\pm}(\oO_X) \quad \mbox{ if and only if }
 \rk(v')=1, 
 \end{align}
 by (\ref{ass1}) and (\ref{ass2}). 
 For $v_1, \cdots, v_l \in S_{\varepsilon, v}(\sigma)$, 
  we can take the limit of the combinatorial 
 coefficients, 
 \begin{align*}
 S(\{v_1, \cdots, v_l\}, \sigma_{+}, \sigma_{-})
 &\cneq \lim_{t\to +0}S(\{v_1, \cdots, v_l\}, \sigma_{t}, \sigma_{-t}), \\
 U(\{v_1, \cdots, v_l\}, \sigma_{+}, \sigma_{-})
 &\cneq \lim_{t\to +0}U(\{v_1, \cdots, v_l\}, \sigma_{t}, \sigma_{-t}).
 \end{align*}
 \begin{step}\label{step1}
(i) If $v_i \in W$ for all $i$, we have 
\begin{align}\label{10}
S(\{v_1, \cdots, v_l\}, \sigma_{+}, \sigma_{-})= 
\left\{ \begin{array}{cc}1, & l=1, \\
0, & l\ge 2. \end{array} \right. 
\end{align}

(ii)
Suppose that
there is $1\le e\le l$ such that $\rk(v_e)=1$
and $v_i \in \Gamma_{0}$ for $i\neq e$. 
If
$S(\{v_1, \cdots, v_l\}, \sigma_{+}, \sigma_{-})\neq 0$, then 
$v_i \in W$ for all $i\neq e$.
Moreover in this case, we have $e=1$ or $2$ and 
\begin{align}\label{l-e}
S(\{v_1, \cdots, v_l\}, \sigma_{+}, \sigma_{-})=(-1)^{l-e}.
\end{align}
\end{step}
\begin{proof}
(i) By the assumption (\ref{extra}), 
$\phi_{\sigma_t}(v)\ge \phi_{\sigma_t}(v')$ holds
for $v, v'\in W$ if 
and only if this holds at $t=0$. 
Then (\ref{10}) follows easily from the definition of 
$S(\{v_1, \cdots, v_l\}, \sigma_{+}, \sigma_{-})$. 
(See the proof of~\cite[Theorem~4.5]{Joy1}.)

(ii) Suppose that there is $2\le i<e$ such that 
 \begin{align}\label{st1}
 \phi_{+}(v_1)>\cdots >\phi_{+}(v_i)\le \phi_{+}(v_{i+1}),
 \end{align}
 holds. 
 By the definition of 
 $S(\{v_1, \cdots, v_l\}, \sigma_{+}, \sigma_{-})$
 and $\phi_{+}(v_1)>\phi_{+}(v_2)$, we have 
 \begin{align}\label{hold}
 \phi_{-}(v_1) \le \phi_{-}(v_2+\cdots +v_l)=\phi_{-}(v_e). 
 \end{align}
 On the other hand, by 
 $\phi_{+}(v_{i-1})>\phi_{+}(v_i)\le \phi_{+}(v_{i+1})$
 we have 
 \begin{align*}
 \phi_{-}(v_1+\cdots +v_{i-1}) \le 
 \phi_{-}(v_{i}+\cdots +v_l)=\phi_{-}(v_e), \\
 \phi_{-}(v_1+\cdots +v_i)> \phi_{-}(v_{i+1}+\cdots +v_l)=\phi_{-}(v_e),
 \end{align*}
 which implies 
  $\phi_{-}(v_i)>\phi_{-}(v_e)$
  and contradicts that (\ref{hold}). 
  Hence a sequence (\ref{st1}) does not happen. 
  Similarly a sequence 
  \begin{align*}
 \phi_{+}(v_1)\le \cdots \le \phi_{+}(v_i)> \phi_{+}(v_{i+1}),
 \end{align*}
  does not happen for $2\le i<e$. 
  Therefore we have two possibilities, 
  \begin{align}\label{two1}
  \phi_{+}(v_1)>\cdots >\phi_{+}(v_{e-1})>\phi_{+}(v_e), \\
  \label{two2}
  \phi_{+}(v_1)\le \cdots \le \phi_{+}(v_{e-1}) \le \phi_{+}(v_e).
  \end{align}
  In the case of (\ref{two1}), 
  (resp.~(\ref{two2}),)
  the definition of 
  $S(\{v_1, \cdots, v_l\}, \sigma_{+}, \sigma_{-})$ implies
  that 
   $\phi_{-}(v_1) \le \phi_{-}(v_e)$, 
  (resp.~$\phi_{-}(v_1)>\phi_{-}(v_{e})$.)
  Hence $v_1 \in W$,
   and by our assumptions (\ref{ass1}) and (\ref{ass2}), 
  the inequality (\ref{two2}) 
  does not happen, hence we have (\ref{two1}). 
  Suppose that $e\ge 3$. Then we have
  $\phi_{-}(v_1+v_2)\le \phi_{-}(v_e)$, hence 
  $v_1+v_2 \in W$. Since $v_1 \in W$, we also have 
  $v_2 \in W$, which contradicts to $\phi_{+}(v_1)>\phi_{+}(v_2)$. 
  Hence we have either $e=1$ or $e=2$. 
   A similar argument 
  also shows $v_i \in W$ for $i>e$. 
  
  Conversely if $e=1$ or $e=2$ and $v_i \in W$ for
  $i\neq e$, 
  it is easy to see that   
 one of
  (\ref{eith1}) or (\ref{eigh2})
  holds for each $i$. 
  Hence (\ref{l-e}) holds by the definition
  of $S((\{v_1, \cdots, v_l\}, \sigma_{+}, \sigma_{-})$.   
  \end{proof}
\begin{step}\label{step2}
Take $v_1, \cdots, v_l\in S_{\varepsilon, v}(\sigma)$ 
satisfying  
$\rk(v_{e})=1$ and $v_i \in \Gamma_{0}$ for 
$i\neq e$. 
 Then 
 $U(\{v_1, \cdots, v_l\}, \sigma_{+}, \sigma_{-})$
 is non-zero only if $v_i \in W$ for $i\neq e$. 
 In this case, we have 
 \begin{align}\label{Uex}
U(\{v_1, \cdots, v_l\}, \sigma_{+}, \sigma_{-})
=\frac{(-1)^{l-e}}{(e-1)!(l-e)!}.
\end{align}
\end{step}
\begin{proof}
Let 
\begin{align*}
&\psi \colon \{1, \cdots, l\} \to \{1, \cdots, l'\}, \\
&\xi \colon \{1, \cdots, l'\} \to \{1, \cdots, l''\},
\end{align*}
be maps which appear in a non-zero term of 
(\ref{def:U}). 
By (\ref{tau})
and (\ref{iff}), we have $l''=1$. 
Also by Step~\ref{step1} (ii), 
$\psi(e)$ is either $1$ or $2$, and
$\psi(e)=1$ (resp.~$\psi(e)=2$) is 
equivalent to $e=1$, (resp.~$e\ge 2$,)
by (\ref{tau0}) and (\ref{iff}), 
and $\psi^{-1}\psi(e)=\{e\}$ holds. 
By Step~\ref{step1} (ii), each $w_i$ 
defined by (\ref{wi}) is contained in $W$
for $i\neq \psi(e)$. Noting the condition
(\ref{tau0}), 
we conclude that $v_i \in W$ for any $i\neq e$. 
  
Now
substituting (\ref{l-e}) yields, 
\begin{align}\label{Uex2}
&U(\{v_1, \cdots, v_l\}, \sigma_{+}, \sigma_{-})
= \\
&\qquad \qquad 
\sum_{\psi \colon \{e+1, \cdots, l\} \to 
\{1, \cdots, l'\}}
(-1)^{l'}\frac{1}{(e-1)!}
\prod_{b=1}^{l'}\frac{1}{\lvert \psi^{-1}(b)\rvert!},
\end{align}
where $\psi$ are non-decreasing surjective maps. 
For a fixed $l$, we have  
\begin{align}\label{element}
\sum_{\psi \colon \{1, \cdots, l\}
 \to \{1, \cdots, l'\}}(-1)^{l-l'}
 \prod_{b=1}^{l'}
 \frac{1}{\lvert \psi^{-1}(b)\rvert !}
=\frac{1}{l!},
\end{align}
for non-decreasing surjective maps $\psi$. 
(See~\cite[Proposition~4.9]{Joy4}.)
Hence we obtain (\ref{Uex}). 
\end{proof}
\begin{step}
For $v=(-n, -\beta, 1) \in \Gamma$, 
we have the formula, 
\begin{align}\label{DTform}
\widehat{\DT}_{n, \beta}(\sigma_{-})=
\sum_{\begin{subarray}{c}
 n_1+ \cdots +n_l=n, \\
\beta_1+\cdots +\beta_l=\beta, \\
-(n_i, \beta_i)\in W, \ 1\le i\le l-1.
\end{subarray}}
\frac{1}{(l-1)!}\prod_{i=1}^{l-1}n_i \widehat{N}_{n_i, \beta_i}
\widehat{\DT}_{n_l, \beta_l}(\sigma_{+}).
\end{align}
\begin{proof}
We apply the formula (\ref{Trans2}) for 
$\tau'=\sigma_{-t}$, $\tau=\sigma_{t}$ for $0<t \ll 1$. 
Take $v_1, \cdots, v_l \in S_{\varepsilon, v}(\sigma)$ which 
appear in a non-zero term of (\ref{Trans2}). 
By Step~\ref{step2},
Lemma~\ref{Sfin} (ii) and Remark~\ref{rk},  
there is $1\le e \le l$ such that 
$v_i \in W$ for $i\neq e$ and $\rk(v_e)=1$. 
Let us write $v_i=(-n_i, -\beta_i, 0)$
for $i\neq e$ and $v_e=(-n_e, -\beta_e, 1)$. 
Since we have 
\begin{align*}
\chi(v_i, v_j)=0, \quad (i, j \neq e), 
\quad \chi(v_i, v_e)=n_i,
\quad \chi(v_e, v_i)=-n_i, 
\end{align*}
an oriented graph $G$ which appears in (\ref{Trans2}) 
is of the following form, 
$$
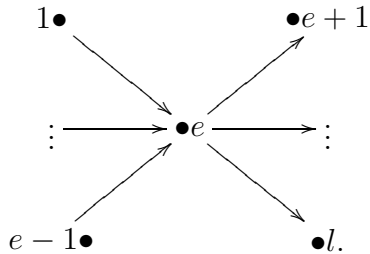
\xymatrix{
1 \bullet \ar[dr] &   & \bullet e+1 \\
\vdots \ar[r] & \bullet e \ar[ur] \ar[dr] \ar[r] &\vdots \\
e-1 \bullet \ar[ur] & & \bullet l.
}$$
Hence substituting (\ref{Uex}) to (\ref{Trans2}), we obtain 
\begin{align*}
\widehat{\DT}_{n, \beta}(\sigma_{-})=\sum_{1\le e\le l}
\sum_{\begin{subarray}{c}
 n_1+ \cdots +n_l=n, \\
\beta_1+\cdots +\beta_l=\beta, \\
-(n_i, \beta_i)\in W, \ i\neq e.
\end{subarray}}
&\frac{1}{2^{l-1}(e-1)!(l-e)!} \\
&\quad \prod_{i\neq e}n_i \widehat{N}_{n_i, \beta_i}
\widehat{\DT}_{n_e, \beta_e}(\sigma_{+}).
\end{align*}
Noting that 
 $$\sum_{1\le e\le l} \frac{1}{2^{l-1}(e-1)!(l-e)!}=\frac{1}{(l-1)!}, $$
 we obtain the formula (\ref{DTform}).
\end{proof}
\end{step}
Obviously (\ref{DTform}) implies (\ref{thm:gen1}) as expected. 
\end{proof}
\begin{rmk}
Since $\sum_{-(n, \beta)\in W}n\widehat{N}_{n, \beta}x^n y^{\beta}$ belongs to 
$\kakkoT$ by (\ref{T1}), the formulas (\ref{thm:gen1}),
(\ref{thm:gen2}) make sense. 
\end{rmk}
Next we compare $\widehat{\DT}(\sigma)$
and $\widehat{\DT}(\tau)$ for 
two weak stability conditions $\sigma, \tau \in \vV$. 
We introduce the notion of general points in $\vV$. 
\begin{defi}\label{defi:general}
\emph{
We say $\sigma=(Z, \pP) \in \vV$ is \textit{general}
if there is no $v\in \Gamma_{0}$ which 
satisfies $Z(v)\in \mathbb{R}_{>0}Z(\oO_X)$.} 
\end{defi}
For general $\sigma, \tau \in \vV$, take 
a good path, (cf.~Definition~\ref{def:good},) 
$$[0, 1] \ni t \mapsto \sigma_t =(Z_t, \pP_t)
\in \vV, $$
which satisfies $\sigma_0=\sigma$ and $\sigma_1=\tau$. 
For $c\in [0, 1]$, let
$W_c$ be the set, 
$$W_c=\{v\in \Gamma_{0} :
Z_{c}(v)\in \mathbb{R}_{>0}Z_{c}(\oO_X)\}.$$  
For $c\in [0, 1]$, we set $\epsilon(c)=1$ 
(resp.~$\epsilon(c)=-1$,)
if
(\ref{good2}) (resp.~(\ref{good1})) happens at $t=c$. 
As a corollary of Theorem~\ref{main:DT}, we obtain the following. 
\begin{cor}\label{path}
We have the equalities of the generating series, 
\begin{align}\label{cor:gen}
\widehat{\DT}(\tau)&=
\widehat{\DT}(\sigma)\cdot \prod_{\begin{subarray}{c}
-(n, \beta) \in W_c, \\
c \in (0, 1).
\end{subarray}}
\exp(n\widehat{N}_{n, \beta}x^n y^{\beta})^{\epsilon(c)}, \\
\label{cor:gen2}
\widehat{\DT}_{0}(\tau)&=
\widehat{\DT}_{0}(\sigma)\cdot \prod_{\begin{subarray}{c}
-(n, \beta) \in W_c, \\
c \in (0, 1).
\end{subarray}}
\exp(n\widehat{N}_{n, \beta}x^n y^{\beta})^{\epsilon(c)}.
\end{align}
\end{cor}
\begin{proof}
We only show (\ref{cor:gen}), as (\ref{cor:gen2}) is 
similarly proved. 
It is enough to show the equality (\ref{cor:gen})  
after the projection, 
$$\pi_{\lambda}\colon \kakkoS\to 
\kakkoS/\kakkoSl,$$
where $\{S_{\lambda}\}_{\lambda \in \Lambda}$
is given in Assumption~\ref{assum2}. 
By Lemma~\ref{wacham},  
there is a finite number of points 
$$c_0=0 <c_1 <\cdots < c_{k-1}< c_k=1,$$
 such that
 $\pi_{\lambda}\widehat{\DT}({\sigma_{t}})$ 
 is constant on $t\in (c_{i-1}, c_i)$ 
 for $1<i <k-1$, and constant on $[0, c_1)$, 
 $(c_{k-1}, 1]$ since $\sigma$ and $\tau$ are general. 
Applying Theorem~\ref{main:DT}
at each $t=c_i$, we obtain 
$$
\pi_{\lambda}\widehat{\DT}(\tau)=
\pi_{\lambda}\widehat{\DT}(\sigma)\cdot \pi_{\lambda}
\prod_{\begin{subarray}{c}
-(n, \beta) \in W_{c_i}, \\
0<i<k.
\end{subarray}}
\exp(n\widehat{N}_{n, \beta}x^n y^{\beta})^{\epsilon(c_i)}.
$$
On the other hand, since $S\setminus S_{\lambda}$ is a finite set, 
there is only finitely many $c\in (0, 1)$ 
such that 
$$\pi_{\lambda} \prod_{-(n, \beta)\in W_c}
\exp(n\widehat{N}_{n, \beta}x^n y^{\beta})
\neq 1.$$
Also such $c\in (0, 1)$ must be equal to one of $c_i$, 
since otherwise 
$\pi_{\lambda}\widehat{\DT}(\sigma_t)$ is constant 
near $t=c$, 
and contradicts to Theorem~\ref{main:DT}.  
Hence (\ref{cor:gen}) holds after the projection. 
\end{proof}
\begin{rmk}
For $c\in [0, 1]$, set
$f_c=\sum_{-(n, \beta)\in W_c}n\widehat{N}_{n, \beta}x^n y^{\beta}$. 
Then $f_c \in \kakkoT$ and 
$\{f_c\}_{c\in [0, 1]}$ satisfy the condition (\ref{sharp})
since $t\mapsto \sigma_t$ is a good path. Therefore 
the formulas (\ref{cor:gen}), (\ref{cor:gen2}) make sense. 
\end{rmk}
Another corollary is the following. 
\begin{cor}\label{cor:inde}
The series 
\begin{align}\label{frac:DT}
\widehat{\DT}'(\sigma)=
\frac{\widehat{\DT}(\sigma)}{\widehat{\DT_0}(\sigma)}
\in \kakkoS,
\end{align}
does not depend on general 
$\sigma\in \vV$. 
\end{cor}
\begin{proof}
This follows immediately from Assumption~\ref{assum} and
Corollary~\ref{path}. 
\end{proof}

\begin{rmk}\label{count0}
In the proof of Theorem~\ref{main:dtpt}, 
we can apply Corollary~\ref{path} and obtain
the following formula, 
\begin{align*}
\widehat{\DT}_0(X)&=\widehat{\DT}_{0}(\sigma_{\xi}) \\
&=\prod_{n>0}\exp(n\widehat{N}_{n, 0}x^n)\widehat{\DT}_{0}(\sigma_{\xi'}) \\
&=\prod_{n>0}\exp(n\widehat{N}_{n, 0}x^n).
\end{align*}
On the other hand, we know that $\widehat{\DT}_{0}(X)=M(x)^{\chi(X)}$. 
These equalities give a calculation of $\widehat{N}_{n, 0}$. 
An easy computation shows, 
$$\widehat{N}_{n, 0}=\sum_{r|n}\frac{\chi(X)}{r^2}.$$
\end{rmk}

\section{Some technical lemmas}
\label{sec:tech}
\subsection{Proof of Proposition~\ref{t-str}.}
\begin{proof}
We first show the following lemma. 
\begin{lem}\label{cok}
(i) For $F\in \langle E, \aA'\rangle_{\ex}\subset \aA$, 
let $u\colon E\to F$ be a non-zero morphism in 
$\aA$. Then $\Cok(u) \in \aA$ is contained in 
$\langle E, \aA'\rangle_{\ex}$ and $u$ is injective in 
$\aA$. 

(ii)
For $F\in \langle E, \aA'\rangle_{\ex}\subset \aA$
and $G\in \aA'$, let
$u\colon G \to F$ be a non-zero morphism in $\aA$.
Then $\Cok(u)\in \aA$ is contained in
$\langle E, \aA'\rangle_{\ex}$. 
\end{lem}
\begin{proof}
(i) By the condition (\ref{cond}), 
there is an exact sequence in $\aA$, 
\begin{align}\label{exse}
0 \lr F_1 \stackrel{u_1}{\lr} F \stackrel{u_2}{\lr} F_2 \lr 0, 
\end{align}
with $F_i \in \langle E, \aA' \rangle_{\ex}$ 
which satisfy the following. 
\begin{itemize}
\item The composition $u_2 \circ u=0$. Hence 
$u$ factorizes as $E \stackrel{u'}{\to}F_1 \stackrel{u_1}{\to}F$. 
\item There is a surjection $F_1 
\stackrel{u_3}{\twoheadrightarrow} E$ in $\aA$ 
such that $\Ker (u_3) \in \langle E, \aA'\rangle_{\ex}$ and 
the composition $u_3 \circ u'$ is non-zero. 
\end{itemize}
Since $\End(E)=\mathbb{C}$, the map $u'$ is split 
injective, hence $u$ is also injective. Also we have 
$\Cok(u')\cong \Ker(u_3) \in
\langle E, \aA'\rangle_{\ex}$, hence the exact sequence in $\aA$, 
$$0 \lr \Cok(u') \lr \Cok(u) \lr F_2 \lr 0, $$
shows $\Cok(u) \in \langle E, \aA'\rangle_{\ex}$. 

(ii) We may assume 
that $u$ is injective since $\aA'$ is closed under quotients
in $\aA$. 
Since $F\in \langle E, \aA'\rangle_{\ex}$, 
there is a filtration in $\aA$
$$A_0 \subset A_1 \subset \cdots A_{l-1}\subset A_l =F, $$
such that each subquotient
$A_i/A_{i-1}$ is either isomorphic to $E$ or contained in $\aA'$. 
We call the smallest such $l$ the length of $F$. We show the
claim by the induction on $l$. If $l=1$, then $F\in \aA'$
by the condition (\ref{cond}). 
Since $\aA' \subset \aA$ is closed under quotients, 
we have $\Cok(u) \in \aA'$. 
Assume that $l>1$. Then there is an exact sequence (\ref{exse})
such that the length of $F_i$ are strictly
 smaller than $l$. Let
$G_2$ be the image of the composition in $\aA$,  
$$G\stackrel{u}{\lr}F \stackrel{u_2}{\lr} F_2,$$
and $G_1$ the kernel of $G\twoheadrightarrow G_2$. 
We obtain the morphism of exact sequences in $\aA$, 
$$\xymatrix{
0 \ar[r] & G_1 \ar[r] \ar[d] & G \ar[r] \ar[d]^{u} &
G_2 \ar[r]\ar[d] & 0 \\
0 \ar[r] & F_1 \ar[r] & F \ar[r] & F_2 \ar[r] & 0.
}$$
Note that $G_i \in \aA'$ and each vertical arrows are injective 
in $\aA$. Hence we have the exact sequence in $\aA$, 
$$0 \lr F_1/G_1 \lr \Cok(u) \lr F_2/G_2 \lr 0.$$
By the induction hypothesis, 
we have $F_i/G_i \in \langle E, \aA'\rangle_{\ex}$. 
Therefore we have $\Cok(u) \in \langle E, \aA'\rangle_{\ex}$. 
\end{proof}
\textit{Proof of Proposition~\ref{t-str}}:

We show that $\aA_E \cneq \dD_E \cap \aA$
is the heart of a bounded t-structure on $\dD_E$, and 
written as $\aA_E=\langle E, \aA'\rangle_{\ex}$. To show this, 
it is enough to show that for any $F\in \dD_E$, 
we have 
\begin{align}\label{property:tst}
\hH_{\aA}^{i}(F)\in \langle E, \aA'\rangle_{\ex}, \quad i\in \mathbb{Z}.
\end{align}
Here we denote by 
$$\hH_{\aA}^{i}\colon \dD \ni F \longmapsto \hH_{\aA}^i(F) \in \aA,$$
the $i$-th cohomology functor with respect to the t-structure
with heart $\aA$. 
Noting that $\aA'$ is the heart of a bounded t-structure 
on $\dD'$, the triangulated category 
$\dD_E$ is also written as $\langle E, \aA'\rangle_{\tr}$. 
Hence any object $F\in \dD_E$ is written as a successive 
extensions by objects $E[i']$ and $G[i'']$ for 
$G\in \aA'$ and $i', i''\in \mathbb{Z}$. 
As in the proof of Lemma~\ref{cok} (ii), 
we show (\ref{property:tst}) by 
the length of such an extension. 
Suppose that $F\in \dD_E$ satisfies (\ref{property:tst}), 
and take a distinguished triangle 
$$G \lr F \lr H \lr G[1], $$
with $G=E$ or $G\in \aA'$. 
By the induction argument, it is enough 
to show that $H$ satisfies (\ref{property:tst}).
Taking the long exact sequence associated to 
$\hH_{\aA}^{\bullet}(\ast)$, we
have $\hH_{\aA}^{i}(F) \cong \hH_{\aA}^{i}(H)$ 
for $i\neq -1, 0$ and the 
exact sequence in $\aA$, 
$$0 \to \hH_{\aA}^{-1}(F) \to \hH_{\aA}^{-1}(H) \to 
G \stackrel{u}{\to} \hH_{\aA}^{0}(F) \to \hH_{\aA}^{0}(H)\to 0.$$
If $u=0$, then obviously $H$ satisfies (\ref{property:tst}). 
Otherwise we have 
$$\Ker(u) \in \aA', \quad \imm (u) \in \aA', \quad 
\Cok(u) \in 
\langle E, \aA'\rangle_{\ex},$$
 by 
Lemma~\ref{cok}. Therefore  
$H$ satisfies (\ref{property:tst}) also in this case. 
\end{proof}

\subsection{Proof of Lemma~\ref{const:stab2}}
We first show the following lemma. 
\begin{lem}\label{Anoet}
The abelian category $\aA_X$ is noetherian. 
\end{lem}
\begin{proof}
We take a chain of surjections in $\aA_X$, 
\begin{align}\label{surj}
E_0 \twoheadrightarrow E_1 \twoheadrightarrow \cdots E_j 
\twoheadrightarrow E_{j+1} \twoheadrightarrow \cdots.
\end{align}
By Lemma~\ref{percoh}, we have 
$\rk(E)\ge 0$ and $\ch_2(E)\cdot \omega \ge 0$ 
for a fixed ample divisor $\omega$ on $X$. 
Hence we may assume that $\rk(E_i)$ and 
$\ch_2(E_i)\cdot \omega$ are constant for all $i$. 
We have exact sequences, 
\begin{align}\label{Q}
0 \lr Q_j[-1] \lr E_j \lr E_{j+1} \lr 0, 
\end{align}
where $Q_j$ are 0-dimensional sheaves. 
The long exact sequence associated to the standard 
t-structure on $D^b(\Coh(X))$ shows that the 
induced morphisms 
$\hH^1(E_{j})\to \hH^1(E_{j+1})$ 
are surjections of sheaves, hence 
we may assume that $\hH^1(E_{j})\cong \hH^1(E_{j+1})$. 
Then the exact sequence (\ref{Q}) induces the sequence, 
\begin{align}\label{dual}
\hH^0(E_0) \subset \hH^0(E_{1}) \subset \cdots \subset 
\hH^0(E_{j}) \subset \cdots \subset 
\hH^0(E_0)^{\vee \vee}.
\end{align}
Since $\Coh(X)$ is noetherian,
the above sequence terminates. 
\end{proof}
\textit{Proof of Lemma~\ref{const:stab2}}:
\begin{ssstep}
The pair $\sigma_{\xi}=(Z_{\xi}, \aA_X)$
satisfies the Harder-Narasimhan property. 
\end{ssstep}
\begin{proof}
Let $\fF$ be the full subcategory of 
$\Coh_{\le 1}(X)[-1]$, 
\begin{align}\label{pure}
\fF=\{F[-1] : F\mbox{ is a pure 1-dimensional sheaf. }\}.
\end{align}
If we set $\tT=\{ T\in \aA_X : \Hom(T, \fF)=0\}$, 
then the pair $(\tT, \fF)$ is a torsion pair on $\aA_X$. 
In fact for any $E\in \aA_X$, 
we have the exact sequence in $\aA_X$, 
$$0 \lr Q[-1] \lr \hH^1(E)[-1] \lr F[-1] \lr 0, $$
where $Q$ is a 0-dimensional sheaf
 and $F[-1] \in \fF$. 
Let $T$ be the kernel of the 
surjection in $\aA_X$, $E\to \hH^1(E)[-1] \to F[-1]$. 
Since $\Hom(\hH^0(E), \fF)=0$, 
we have $\Hom(T, \fF)=0$, i.e. $T\in \tT$. 
The exact sequence $0 \to T \to E \to F[-1]\to 0$
gives the desired decomposition (\ref{fits}). 

It is easy to see that any $F[-1] \in \fF$ is 
$Z_{\xi}$-semistable with 
$\arg Z_{\xi}(T)>\arg Z_{\xi}(F[-1])$
for any non-zero $T\in \tT$. 
Also applying the same argument of~\cite[Lemma~2.27]{Tolim}, 
an object $E\in \tT$ is $Z_{\xi}$-semistable if and 
only if for any exact sequence 
$$0\lr A \lr E \lr B \lr 0, $$
with $A, B \in \tT$, we have $\arg Z_{\xi}(A) \le \arg Z_{\xi}(B)$. 
By Lemma~\ref{Anoet} and 
the proof of Proposition~\ref{suHN}, 
it is enough to show that 
there is no infinite sequence of subobjects 
in $\tT$, 
\begin{align}\label{locfin}
 \cdots \subset E_{j+1} \subset E_j \subset \cdots 
 \subset E_2 \subset E_1.
 \end{align}
 (cf.~the proof of~\cite[Theorem~2.29]{Tolim}.)
 Suppose that such a sequence exists. 
 We may assume that $\rk(E_i)$ and $\ch_2(E_i)\cdot \omega$
 are constant for all $i$, 
 hence $E_{0}/E_{j+1}=Q_j[-1]$ where $Q_j$ is a 
 0-dimensional sheaf. 
 Taking the long exact sequence of cohomology, we obtain the 
 sequence of surjections of sheaves, 
 \begin{align}\label{Q_j}
  \cdots \twoheadrightarrow
   Q_{j+1} \twoheadrightarrow Q_j \twoheadrightarrow \cdots 
  \twoheadrightarrow 
  Q_{1} \twoheadrightarrow Q_{0}.
  \end{align}
  Since $\hH^1(E_0)$ is 0-dimensional and 
  we have the surjections $\hH^1(E_0) \twoheadrightarrow Q_j$
for all $j$, the length of 
 $Q_j$ is bounded above. 
 This implies (\ref{Q_j}) terminates, hence (\ref{locfin}) also terminates.
 \end{proof}
 \begin{ssstep}
 Let $\{\pP_{\xi}(\phi)\}_{\phi \in \mathbb{R}}$
  be the slicing corresponding to the pair 
 $\sigma_{\xi}=(Z_{\xi}, \aA_X)$
 via Proposition~\ref{prop:corr}. 
 Then $\{\pP_{\xi}(\phi)\}_{\phi \in \mathbb{R}}$ 
 is of locally finite. 
 \end{ssstep}
 \begin{proof}
 Since $\aA_X$ is noetherian, it is enough 
 to show that there is $\eta>0$ such that 
 $\pP_{\xi}((\phi-\eta, \phi+\eta))$ is 
 artinian for any $\phi \in \mathbb{R}$
 with respect to strict monomorphisms. 
 Let $\phi_i=\frac{1}{\pi}\arg z_i \in (1/2, 1)$. 
 By the construction of $Z_{\xi}$, 
 it is enough to show that 
 $\pP_{\xi}(1/2)$ and 
 $\langle \pP_{\xi}(\phi_0), \pP_{\xi}(\phi_1)\rangle_{\ex}$
 are artinian.
 It is easy to see that $\pP_{\xi}(1/2)$ 
 coincides with $\fF$, where $\fF$ is given by (\ref{pure}). 
 Suppose that there is an infinite sequence of strict monomorphisms in 
 $\pP_{\xi}(1/2)$, 
 \begin{align}\label{locfin2}
 \cdots \subset E_{j+1} \subset E_j \subset \cdots 
 \subset E_2 \subset E_1.
 \end{align}
 Since each $E_j$ is a 1-dimensional sheaf, 
 we have
 $\ch_2(E_{j+1})\cdot \omega \le \ch_2(E_j)\cdot \omega$ for 
 an ample divisor $\omega$, and 
 $\ch_2(E_j)\cdot \omega=0$ if and only if $E_j=0$. 
 Therefore (\ref{locfin2}) terminates. 
 The artinian condition of 
  $\langle \pP_{\xi}(\phi_0), \pP_{\xi}(\phi_1)\rangle_{\ex}$
  follows from the same argument to show the 
  termination of (\ref{locfin}).  
     \end{proof}
 \begin{ssstep}
 The pair $\sigma_{\xi}=(Z_{\xi}, \pP_{\xi})$
 satisfies the support property (\ref{support}). 
  \end{ssstep}
 \begin{proof}
 Let $E\in \aA_X$ be a non-zero object 
 with $\cl(E)=(-n, -\beta, r)$. 
 We introduce an usual Euclid 
 norm on $\mathbb{H}_{0}\otimes \mathbb{R}=\mathbb{R}$
 and $\mathbb{H}_2 \otimes \mathbb{R}=\mathbb{R}$. 
 We have 
 \begin{align*}
 \frac{\lVert E \rVert}{\lvert Z(E) \rvert}
 =\left\{ \begin{array}{cc}
  \lvert z_1 \rvert,  & r>0, \\
  \frac{\lVert \beta \rVert}{\omega \cdot \beta}, & r=0, \ \beta\in \NE(X), \\
  \lvert z_0 \rvert, & r=\beta=0, \ n>0.
 \end{array}
 \right. 
 \end{align*}
 Since $\beta$ is effective or zero, the 
 above description immediately implies 
 the support property.  
 \end{proof}

\subsection{Proof of Lemma~\ref{check}}
\begin{proof}
The first, second, third and the last 
conditions are obviously satisfied. We check 
other three conditions. 
Recall the heart of a bounded t-structure 
$\aA_X \subset \dD_X$ given in (\ref{tst1}).  
\begin{sstep}\label{sstep1}
For $v\in \Gamma$ with $\rk(v)=1$ or $v\in \Gamma_{0}$, 
the stack of objects 
$$\oO bj^{v}(\aA_X)\subset \mM_{0}, $$
is an open substack of $\mM_0$. 
\end{sstep}
\begin{proof}
If $\rk(v)=0$, then $\oO bj^{v}(\aA)$ is the stack 
of coherent sheaves $E\in \Coh_{\le 1}(X)$
of numerical type $v$, 
and the result is well-known. Suppose that $\rk(v)=1$
and 
let $\eE \in D^b(X\times S)$ be an $S$-valued 
point of $\mM_0$. We assume that $S$ is connected and 
there is $s\in S$ such that  
$\eE_s \cneq \dL i_{s}^{\ast}\eE \in \aA_X$
with $\cl(\eE_s)=v$, where $i_s \colon X\times \{s\} 
\hookrightarrow X\times S$ is the inclusion. 
It is enough to show that the locus 
\begin{align}\notag
S^{\circ}\cneq \{ s'\in S : \eE_{s'}\in \aA_X\}, 
\end{align}
is open in $S$. 
Note that the stack of objects in $E\in \Coh^{\dag}(X)[-1]$,
(cf.~(\ref{dag}),)
with $\det (E)=\oO_X$ 
is open in $\mM_{0}$,
(cf.~\cite[Lemma~3.14]{Tolim},)
hence we may assume that $\eE_{s'}\in \Coh^{\dag}(X)[-1]$
for all $s'\in S$. 
As in the proof of Lemma~\ref{firstshow} (i), any 
object $E\in \aA_X$ with $\rk(E)=1$ is given 
by an extension, 
$$I_C \lr E \lr F[-1], $$
where $I_C$ is the ideal sheaf of $C\subset X$
with $\dim C\le 1$ and $F\in \Coh_{\le 1}(X)$. 
Therefore an object $E\in \Coh^{\dag}(X)[-1]$ with 
$\det(E)=\oO_X$ and $\rk(E)=1$ is contained in $\aA_X$
if and only if $\hH^{0}(E)$ is torsion free. 
First we show the case that $S$ is a smooth curve. 
We have the spectral sequence, 
$$E_{2}^{p, q}=\tT or_{-p}^{\oO_{X\times S}}(\hH^{q}(\eE), 
\oO_{X\times \{s\}}) \Rightarrow \hH^{p+q}(\eE_{s}).$$
Since $E_{2}^{p, q}=0$ for $p\le -2$ or $p\ge 1$, the above 
spectral sequence degenerates at $E_2$-terms. 
Therefore $E_{2}^{-1, 0}=0$, and this implies that
$\hH^{0}(\eE)$ is flat over $S$, and 
we have the exact sequence, 
$$0 \lr \hH^{0}(\eE)_{s} \lr \hH^{0}(\eE_s) \lr 
\tT or_{1}^{\oO_{X\times S}}(\hH^{1}(\eE), \oO_{X\times \{s\}})
\lr 0.$$
Since $\hH^{0}(\eE_s)$ is torsion free by 
$\eE_s \in \aA_X$, 
the sheaf $\hH^{0}(\eE)_{s}$ is also torsion free
by the above exact sequence.  
Since $\hH^{0}(\eE)$ is flat over $S$, there is 
an open neighborhood $s\in U$ such that 
$\hH^{0}(\eE)_{s'}$ is torsion free for $s'\in U$. 
By the generic flatness, we have 
$\hH^{0}(\eE_{s'})\cong \hH^{0}(\eE)_{s'}$ for 
$s\neq s' \in U$ by shrinking $U$ if necessary. 
Therefore $U\subset S^{\circ}$ and $S^{\circ}$ is open in $S$. 

In general, we can show the openness of $S^{\circ}$ as follows. 
Let $V\subset S$ be an open subset on which 
$\hH^j(\eE)$ is flat for all $j$. Then we have 
$\hH^j(\eE)=0$ unless $j=0, 1$,
$\hH^{0}(\eE_{s'})=\hH^{0}(\eE)_{s'}$
for any $s'\in V$, and $S^{\circ}\cap V$ is open in $V$. 
Also by the result for smooth curves, we know that 
$S^{\circ}$ is dense in $S$ in Zariski topology, 
hence $S^{\circ}\cap V$ is non-empty. We apply the 
same argument for the object $\dL i^{\ast}\eE$, 
where $i$ is the inclusion, 
$$i\colon S\setminus (S^{\circ}\cap V)
\hookrightarrow S.$$
By the noetherian induction, we conclude that
$S^{\circ}$ is open in $S$.

\end{proof}
\begin{sstep}\label{sstep2}
Take $\sigma_{\xi}=(Z_{\xi}, \aA_X) \in \vV_X$ and
$v\in \Gamma$ with $\rk(v)=1$ or $v\in \Gamma_{0}$. 
Then the substack 
$$\mM^v(\sigma_{\xi}) \subset \oO bj^{v}(\aA_X), $$
is an open substack and it is of
 finite type over $\mathbb{C}$. 
\end{sstep}
\begin{proof}
If $v\in \Gamma_0$, 
then 
$\mM^{v}(\sigma_{\xi})$ is the moduli stack 
of 0-dimensional sheaves, 
and the result is well-known. 
Suppose that $\rk(v)=1$. By 
Step~\ref{sstep1} and
the argument
of~\cite[Theorem~3.20]{Tst3},
it is enough to show the boundedness of 
$\sigma_{\xi}$-semistable objects of numerical type $v$.
For a $\sigma_{\xi}$-semistable object $E \in \aA_X$, 
consider the exact sequence (\ref{useseq}). 
For an effective class $\beta \in N_1(X)$, we set 
$m(\beta)$ as
\begin{align}\label{mbeta}
m(\beta)=\inf\{ \ch_3(\oO_C) : \dim C=1 \mbox{ with }
[C]=\beta\}.
\end{align}
It is well-known that $m(\beta)>-\infty$,
(cf.~\cite[Lemma~3.10]{Tolim},)
hence
 the length of $Q$ in (\ref{useseq})
 is bounded above.
Since the set of ideal sheaves with a fixed numerical 
class is bounded, the object $E$ is contained in a 
bounded family. 
\end{proof}
\begin{sstep}\label{sstep3}
There are subsets $0\in T\subset S \subset N_{\le 1}(X)$
which satisfy Assumption~\ref{assum2}. 
\end{sstep}
\begin{proof}
We set $S$ and $T$ to be 
 \begin{align}\label{good}
 S&\cneq \{ (n, \beta) \in N_{\le 1}(X) : 
 \beta \ge 0, \ n\ge m(\beta)\}, \\
 \label{verygood}
 T& \cneq \{(n, \beta) \in N_{\le 1}(X) :
 \beta\ge 0, \ n\ge 0\}.
 \end{align}
Here $\beta \ge 0$ means $\beta$ is effective or 
zero, $m(\beta)$ is given in (\ref{mbeta})
when $\beta$ is effective, and $m(0)=0$. 
We show that $T$, $S$ satisfy Assumption~\ref{assum2}. 
The first condition is obvious.
The second condition follows easily that 
any effective class in $N_1(X)$ can be written as 
finitely many ways as a sum of effective classes. 
The third one follows from the existence of the 
exact sequence (\ref{useseq}). 
As for the last one, let 
$\Lambda$ be the set of pairs $(k, \beta')$ of 
$k\in \mathbb{Z}$ and an effective class $\beta' \in N_1(X)$. 
For $\lambda=(k, \beta')$, we set 
$$S_{\lambda}=\{ (n, \beta) \in S :
n\ge k \mbox{ if }\beta\le \beta'\}.
$$
Here $\beta\le \beta'$ means $\beta'-\beta$ is effective or zero. 
Then $\{S_{\lambda}\}_{\lambda \in \Lambda}$ 
gives the desired family. 
\end{proof}
Assumption~\ref{assum} has been checked by 
Step~\ref{sstep1}, Step~\ref{sstep2} and Step~\ref{sstep3}. 
\end{proof}

\begin{rmk}
If $v\notin \Gamma_0$, then 
$\mM^{v}(\sigma)$ is not necessary of finite type. 
For instance consider the class $v=(0, 2[C], 0)$ 
for a curve $C\subset X$. Then 
$\mM^v(\sigma)$ contains objects 
$\oO_C(D)\oplus \oO_C(-D)$ for arbitrary divisors 
$D\subset C$. Thus $\mM^v(\sigma)$ is not of finite type. 
\end{rmk}

\section{Some results on weak stability conditions}
\label{sec:some}
\subsection{Outline of the proof of Theorem~\ref{thm:stab}.}
\begin{proof}
We first note that 
if two elements of $\Stab_{\Gamma_{\bullet}}(\dD)$, 
$\sigma=(Z, \pP)$ and 
$\tau=(W, \qQ)$ satisfy 
$d(\pP, \qQ)<1$, then $\sigma=\tau$. 
(See~\cite[Lemma~6.4]{Brs1} for the proof.)
In particular the map $\Pi$ is locally injective, hence
it is enough to show 
that $\Pi$ is locally surjective. 
For $\sigma=(\{Z_i\}_{i=0}^{N}, \pP)$, let us 
take a $\sigma$-semistable object $E\in \dD$ 
with $\cl(E)\in \Gamma_{m}\setminus \Gamma_{m-1}$. 
For $\{W_i\}_{i=0}^{N} \in \prod_{i=0}^{N}\mathbb{H}_i^{\vee}$, 
the support property (\ref{support}) implies,
\begin{align}\label{eva}
\left\lvert 1-\frac{W_m([E])}{Z_m([E])} \right\rvert \le C\cdot 
(W_m-Z_m)\left(\frac{[E]}{\lVert [E] \rVert}\right),
\end{align}
for a constant $C>0$. 
For any $0<\varepsilon\ll 1$, we can find an open neighborhood 
$\{Z_i\}_{i=0}^{N}\in U_{\varepsilon} \subset 
\prod_{i=0}^{N}\mathbb{H}_i^{\vee}$ such that 
the RHS of (\ref{eva}) is less than $\sin \pi \varepsilon$
for any $\{W_i\}_{i=0}^{N}\in U_{\varepsilon}$. 
In particular $W_m([E]) \neq 0$ for such 
$\{W_i\}_{i=0}^{N}$, and we have 
\begin{align}\label{arg}
\lvert \arg W_m([E])-\arg Z_m([E]) \rvert <\pi \varepsilon.
\end{align}
The above condition (\ref{arg}) is enough to
apply the same proof of~\cite[Theorem~7.1]{Brs1}
to show
 the existence of $\qQ \in \Slice(\dD)$ 
satisfying $d(\pP, \qQ)<\varepsilon$ and 
(\ref{phase}) holds for 
the pair $\tau=(\{W_i\}_{i=0}^{N}, \qQ)$.
Here we just describe how to construct $\qQ$, and 
leave the detail to the reader to check that
the proof of~\cite[Theorem~7.1]{Brs1} works
in our situation. 
For $\phi \in \mathbb{R}$
and $a, b \in \mathbb{R}$, 
a quasi-abelian category 
$\pP((a, b))$ 
is called thin and envelopes $\phi$ if 
$b-a<1-2\varepsilon$ and $a+\varepsilon\le \phi \le b-\varepsilon$. 
Then $W=\{W_i\}_{i=0}^{N}$ determines a map, 
$$W\colon \pP((a, b)) \ni E \longmapsto \arg W(E) \in (\pi(a-\varepsilon), 
\pi(b+\varepsilon)).$$ 
 The subcategory $\qQ(\phi)\subset \dD$ is defined by 
$W$-semistable objects $E\in \pP((a, b))$
with phase $\phi$, i.e. 
$E\in \qQ(\phi)$ if and only if 
for any exact sequence in $\pP((a, b))$, 
$$0 \lr F \lr E \lr G \lr 0, $$
we have 
$\arg W(F) \le \arg W(G)$. 
 The same proof of~\cite[Theorem~7.1]{Brs1} shows 
that $\qQ(\phi)$ does not depend on $a$, $b$, 
and determines a desired slicing on $\dD$. 
It is enough to check that $\tau=(\{W_i\}_{i=0}^{N}, \qQ)$
satisfies the support property. 
Let $F\in \dD$ be a $\tau$-semistable object
with $\cl(F)\in \Gamma_{m}\setminus \Gamma_{m-1}$. 
Since (\ref{eva}) is less than $\sin \pi \varepsilon$ and 
$d(\pP, \qQ)<\varepsilon$, 
 it is easy to see 
 $$\lVert [F] \rVert_{m} 
 \le \frac{1-\sin \pi \varepsilon}{\cos 2\pi \varepsilon}
 C \lvert W_m([F]) \rvert.$$
Hence $\tau$ satisfies the support property (\ref{support}),
 and $\Pi$
is surjective on $U_{\varepsilon}$. 
\end{proof}
\subsection{Proof of Lemma~\ref{lem:cont}}
This follows from a 
following stronger lemma below, by 
setting $\fF=0$ there. 
We will need this stronger version in the 
next paper~\cite{Tcurve2}.
\begin{lem}\label{lem:scont}
Let $\aA$ be the heart of a bounded t-structure on 
$\dD$, and $(\tT, \fF)$ a torsion pair on $\aA$. 
Let $\bB=\langle \fF[1], \tT \rangle_{\ex}$ the 
associated tilting. Let
$$[0, 1) \ni t \longmapsto Z_t \in 
\prod_{i=0}^{N}\mathbb{H}_i^{\vee}, $$
be a continuous map such that
$\sigma_t=(Z_t, \aA)$ for $0<t<1$ and 
$\sigma_0=(Z_0, \bB)$ determine points in 
$\Stab_{\Gamma_{\bullet}}(\dD)$.
Then we have $\lim_{t\to 0}\sigma_t=\sigma_0$. 
\end{lem}
\begin{proof}
By Theorem~\ref{thm:stab},
we have a continuous family of
points $\sigma_{t}'=(Z_t, \qQ_t)
\in \Stab_{\Gamma_{\bullet}}(\dD)$
 for $0\le t\ll 1$
 such that $\sigma_{0}'=\sigma_0$. 
It is enough to show that $\sigma_{t}'=\sigma_t$
for such $t$, i.e. 
$\qQ_{t}((0, 1])=\aA$. This 
follows from the inclusion 
\begin{align}\label{fol:inc}
\qQ_{t}((0, 1])\subset \aA,
\end{align}
since both are hearts of bounded t-structures on $\aA$. 
To check (\ref{fol:inc}), first note that any
object $E\in \fF[1]$ is contained in 
$\qQ_{0}(1)$, since otherwise $\Imm Z_t(E)>0$ 
for $0<t\ll 1$ contradicting that 
$Z_t$ is a weak stability function on $\aA$
for such $t$. 
Hence we have 
\begin{align}\label{Q_0}
\qQ_0((0, 1))\subset \tT \subset \aA.
\end{align}
Next 
take $0<\phi \le 1$ and  
a quasi-abelian category $\qQ_0((a, b))$ which 
is thin and envelopes $\phi$. 
(See the proof of Theorem~\ref{thm:stab}.)
As in the proof of Theorem~\ref{thm:stab}, objects of
$\qQ_t(\phi)$ consist of $Z_{t}$-semistable objects 
in $\qQ_0((a, b))$. 
If $a< 1$, then we have 
$\qQ_t(\phi) \subset \qQ_0((a, b)) \subset \aA$ 
by (\ref{Q_0}). 
Suppose $a\ge 1$, and take $E\in \qQ_t(\phi)$. 
Noting that $\fF[1]\subset \qQ_0(1)$ and (\ref{Q_0}), 
we have 
$$\hH^{-1}_{\aA}(E)[1] \in \qQ_{0}([1, a)), \quad
\hH^{0}_{\aA}(E)\in \qQ_{0}((b, 1)).$$
Therefore
the following sequence is an exact sequence in $\qQ_0((a, b))$, 
$$\hH^{-1}_{\aA}(E)[1] \lr E \lr \hH^{0}_{\aA}(E).$$
If $\hH^{-1}_{\aA}(E)\neq 0$, then 
 $\Imm Z_{t}(\hH^{-1}_{\aA}(E)[1])<0$ for $0<t\ll 1$
 which contradicts to $Z_t$-semistability of $E$. 
 Hence $\hH^{-1}_{\aA}(E)=0$, i.e. $E\in \aA$
 and (\ref{fol:inc}) holds. 
\end{proof}

\section{Appendix: involving Behrend's constructible functions}
The proof of Conjecture~\ref{conj:DTPT}
should follow if we are able to involve Behrend's
constructible functions
into our argument. Now there is a progress 
toward this direction by Joyce and Song~\cite{JS}, and 
we are able to establish 
the analogue statement of Theorem~\ref{prop:trans} for 
counting invariants of semistable sheaves
(not objects in the derived category) involving 
Behrend functions using their work. 
At the moment the author wrote the first version 
of this paper, still there was a technical gap in applying their 
theory into our context, that is
the derived category version of~\cite[Theorem~5.3]{JS}
 on the descriptions of 
the local moduli spaces of coherent sheaves on Calabi-Yau 3-folds. 
(i.e. the local moduli space of objects in 
the heart of a t-structure,
not necessary of a standard one, should be written as a critical 
locus of some holomorphic function on a smooth analytic space.)
Now it is announced that 
the above problem is solved by Behrend and Getzler~\cite{BG}, 
so we should now be able to apply Joyce and Song's work into
our context. 
In this appendix, we show that Conjecture~\ref{conj:DTPT}
is true using the result of~\cite{JS} and using~\cite{BG}. 

Recall that for any $\mathbb{C}$-scheme $M$, there is a canonical 
constructible function $\nu_{M} \colon M \to \mathbb{Z}$
by Behrend~\cite{Beh}, such that if $M$
has a perfect symmetric obstruction theory, one has 
$$\int_{[M^{\rm{vir}}]}1=\sum_{n\in \mathbb{Z}}n\chi(\nu^{-1}_{M}(n)).$$
Behrend's constructible functions can be also defined 
for Artin $\mathbb{C}$-stacks locally of finite type. 
That is, if $\mM$ is an Artin $\mathbb{C}$-stack 
locally of finite type, there is a unique locally 
constructible 
function $\nu_{\mM} \colon \mM \to \mathbb{Z}$ such that 
if $f\colon M \to \mM$ is a smooth $1$-morphism 
of Artin $\mathbb{C}$-stacks of relative 
dimension $n$, one has $\nu_{M}=(-1)^{n}\nu_{\mM}$. 
(cf.~\cite[Proposition~4.4]{JS}.)

Now let us consider the situation of Paragraph~\ref{subsec:Joyce}.
The Behrend function on $\oO bj(\aA)$ is denoted by $\nu$. 
The relevant invariants are defined in a similar way 
to Definition~\ref{Jinv}
after involving the function $\nu$. 
In order to state this, note that the map 
$$\nu \cdot \colon \hH(\aA) \lr \hH(\aA),$$
sending $[f\colon \xX \to \oO bj(\aA)] \in \hH(\aA)$
to 
$$\sum_{i}i[\nu^{-1}(i)\times_{\oO bj(\aA)}\xX \to \oO bj(\aA)] 
\in \hH(\aA), $$
is well-defined. 
\begin{defi}\emph{
In the situation of Definition~\ref{Jinv}, 
we define $J^v_{\rm{vir}}(\sigma)\in \mathbb{Q}$ as follows. 
}
\begin{itemize}
\item \emph{If $v\in C_{\sigma}(\phi)$ for $0<\phi \le 1$, we define}
\begin{align*}
J^v_{\rm{vir}}(\sigma)\cneq
\lim_{t\to 1}
(t^2-1)\Upsilon(\nu\cdot\epsilon^{v}(\sigma)).
\end{align*}
\item \emph{If $v\in C_{\sigma}(\phi)$ for $1<\phi \le 2$, we 
define $J^v_{\rm{vir}}(\sigma)\cneq J^{-v}_{\rm{vir}}(\sigma)$.}
\item \emph{Otherwise we define $J^v_{\rm{vir}}(\sigma)=0$. }
\end{itemize}
\end{defi}
The wall-crossing formula for the invariants $J^v_{\rm{vir}}(\sigma)$
is proved using the 
results of~\cite{JS} and~\cite{BG}.  
As in the same way of~\cite[Theorem~5.12]{JS}, 
we obtain a map
$$\Psi \colon 
\mathrm{SF}_{\rm{al}}^{\rm{ind}}(\aA) \to C(X).$$
Here $\mathrm{SF}_{\rm{all}}^{\rm{ind}}(\aA)$ is a certain 
Lie subalgebra of $\hH(\aA)$ consisting of 
virtual indecomposable objects~\cite[Paragraph~5.2]{Joy2}, 
and $C(X)$ is the $\mathbb{Q}$-vector space 
with basis of symbols $c_{v}$ for $v\in \Gamma$, 
with Lie bracket
\begin{align}\label{Lie1}
[c_{v}, c_{v'}]=(-1)^{\chi(v, v')}\chi(v, v')c_{v+v'}.
\end{align}
The map $\Psi$ is 
defined by the same way of~\cite[Equation~(71)]{JS}, 
just replacing $\Coh(X)$ by $\aA$. 
The elements $\epsilon^{v}(\sigma)\in \hH_{\ast}(\aA)$
in Definition~\ref{edel}
are contained in $\mathrm{SF}_{\rm{all}}^{\rm{ind}}(\aA)$,
and if $\epsilon^{v}(\sigma)$ is written as 
$[[M/\mathbb{G}_m]\stackrel{i}{\hookrightarrow} \oO bj(\aA)]$
where $M$ is a variety with $\mathbb{G}_m$ acting trivially, 
and $i$ is an open immersion of stacks, one has 
$$\Psi(\epsilon^{v}(\sigma))=
-\sum_{i}i\chi(\nu_{M}^{-1}(i))c_{v}.$$
\begin{rmk}
Recall that
 another Lie algebra $\tilde{C}(X)$
 with $\mathbb{Q}$-basis of symbols $\tilde{c}_v$ 
 for $v\in \Gamma$ with Lie bracket
 \begin{align}\label{Lie2}
 [\tilde{c}_{v}, \tilde{c}_{v'}]=\chi(v, v')c_{v+v'}, 
 \end{align}
 is introduced in~\cite[Paragraph~6.5]{Joy2}. 
 There is a Lie algebra map, 
 $$\tilde{\Psi}\colon 
 \mathrm{SF}_{\rm{al}}^{\rm{ind}}(\aA) \to \tilde{C}(X),$$
which is used in~\cite{Joy4} to show  
Theorem~\ref{prop:trans}. 
Note that there is a sign difference
between (\ref{Lie1}) and (\ref{Lie2}). 
This is basically due to the fact that 
the Behrend function 
on a scheme $M$ 
has the value $(-1)^{\dim M}$ if $M$ is smooth. 
\end{rmk}
Now let us consider the 
situation of Assumption~\ref{assum}. 
By the assumption, the stack
$$\oO bj^{v}(\aA)\subset \mM_{0},$$
for $v\in \Gamma$ with $\rk(v)=1$ or $v\in \Gamma_0$ 
is an open substack of $\mM_0$. 
Hence by~\cite{BG}, the stack 
$\oO bj^{v}(\aA)$ is locally 
near $E\in \aA$ written as 
a quotient stack $[U/G]$, where $U$ is the critical 
analytic space of a holomorphic function germ 
$$f\colon \Ext_X^1(E, E)_{0} \lr \mathbb{C}.$$
Here $\Ext_{X}^i(E, E)_{0}$ is the kernel of the trace map 
$\Ext_{X}^i(E, E) \to H^i(\oO_X)$, 
$G$ is an algebraic group with Lie algebra
$\Ext_{X}^{0}(E, E)_{0}$, and 
$f$ is $G$-invariant. Note that this is the 
derived category version of~\cite[Theorem~5.3]{JS}, which 
is the only issue in extending~\cite[Theorem~5.12]{JS}
to the context of the derived category. 
Now the same proof of~\cite[Theorem~5.12]{JS} shows that 
the map $\Psi$ restricted to the subspace
$$\mathrm{SF}_{\rm{al}}^{\rm{ind}}(\aA)_{0}
\oplus \mathrm{SF}_{\rm{al}}^{\rm{ind}}(\aA)_{N}
\subset \mathrm{SF}_{\rm{al}}^{\rm{ind}}(\aA), $$
where $\mathrm{SF}_{\rm{al}}^{\rm{ind}}(\aA)_{\ast}$
is defined by 
$$\mathrm{SF}_{\rm{al}}^{\rm{ind}}(\aA)_{\ast}
=\mathrm{SF}_{\rm{al}}^{\rm{ind}}(\aA) \cap \hH_{\ast}(\aA),$$
is a map satisfying 
\begin{align}\label{Liesat}
[\Psi(\epsilon), \Psi(\epsilon')]=\Psi([\epsilon, \epsilon']),
\end{align}
for $\epsilon\in \mathrm{SF}_{\rm{al}}^{\rm{ind}}(\aA)_{0}$
and $\epsilon'\in \mathrm{SF}_{\rm{al}}^{\rm{ind}}(\aA)_{\ast}$
with $\ast=0, N$. 
\begin{rmk}
More precisely, we are not able to define 
$\mathrm{SF}_{\rm{al}}^{\rm{ind}}(\aA)$
under Assumption~\ref{assum}, as we do not assume that 
 $\oO bj^{v}(\aA)$ is algebraic unless 
$\rk(v)=1$ or $v\in \Gamma_0$. 
However as in Paragraph~\ref{subsec:Joyce}, 
we are able to define the Lie algebra
$\mathrm{SF}_{\rm{al}}^{\rm{ind}}(\aA)_{0}$
and the $\mathrm{SF}_{\rm{al}}^{\rm{ind}}(\aA)_{0}$-module 
$\mathrm{SF}_{\rm{al}}^{\rm{ind}}(\aA)_{N}$,
 and the map preserving Lie-brackets as in (\ref{Liesat}), 
 $$\Psi \colon 
 \mathrm{SF}_{\rm{al}}^{\rm{ind}}(\aA)_{0}
 \oplus \mathrm{SF}_{\rm{al}}^{\rm{ind}}(\aA)_{N}
 \to C(X).$$
\end{rmk}

Then the same argument along with the proof of 
Theorem~\ref{prop:trans} in~\cite[Theorem~6.28]{Joy4} 
yields the following. 
\begin{thm}\label{newJ}
We have the formula, 
\begin{align}
\notag
J^v_{\rm{vir}}(\tau')=&\sum_{\begin{subarray}{c}l\ge 1, \ v_i \in 
S_{\varepsilon, v}(\sigma), \\
v_1+\cdots +v_l=v\end{subarray}}
\sum _{\begin{subarray}{c}
G \text{ \rm{is a connected, simply connected oriented}} \\
\text{\rm{graph with vertex} }\{1, \cdots, l\}, \
\stackrel{i}{\bullet}\to \stackrel{j}{\bullet}\text{ \rm{implies} }
i< j
\end{subarray}} \\
\label{Trans3}
&\quad \frac{1}{2^{l-1}}U(\{v_1, \cdots, v_l\}, \tau, \tau')
\prod_{\stackrel{i}{\bullet} \to \stackrel{j}{\bullet}\text{ \rm{in} }G}
(-1)^{\chi(v_i, v_j)}\chi(v_i, v_j)\prod_{i=1}^{l}J^{v_i}_{\rm{vir}}(\tau).
\end{align} 
\end{thm}
The sign difference between (\ref{Trans3})
and (\ref{Trans2}) is due to
the sign difference between (\ref{Lie1}) and (\ref{Lie2}). 
Next we define the generalized Donaldson-Thomas invariants. 
\begin{defi}
\emph{In the situation and notation of Definition~\ref{def:geDT},
Proposition-Definition~\ref{strict}, 
we define $\DT_{n, \beta}(\sigma)$ and $N_{n, \beta}$
to be
\begin{align*}
\DT_{n, \beta}(\sigma)&\cneq -J_{\rm{vir}}^{(-n, -\beta, 1)}(\sigma)
\in \mathbb{Q}, \\
N_{n, \beta}&\cneq -J_{\rm{vir}}^{(-n, -\beta, 0)}(\tau)\in \mathbb{Q}.
\end{align*}}
\end{defi}
\begin{rmk}
Here we need to change the sign, since the
Behrend function on a scheme $M$ and 
on the quotient stack $[M/\mathbb{G}_m]$ 
have different sign. 
\end{rmk}
\begin{rmk}
As in the proof of Proposition-Definition~\ref{strict}, 
the invariant $N_{n, \beta}$ does not depend on $\tau$. 
\end{rmk}
\begin{rmk}
Suppose that 
$\sigma \in \vV$ is a general point.
(cf.~Definition~\ref{defi:general}.) 
Then there is no strictly $\sigma$-semistable 
objects of numerical type $v$, hence 
$$\epsilon^{v}(\sigma)=\delta^{v}(\sigma)=
[[M/\mathbb{G}_m] \stackrel{i}{\hookrightarrow}\oO bj(\aA)],$$
for an algebraic space
 $M$ with a perfect symmetric obstruction theory
 and $\mathbb{G}_m$ acting trivially, and
$i$ is an open immersion of stacks.
(cf.~\cite{HT2}.)
Therefore we have 
$$\DT_{n, \beta}(\sigma)=\sum_{i}i\chi(\nu_{M}^{-1}(i))
=\int_{[M^{\rm{vir}}]}1.$$
\end{rmk}
The generating series are 
similarly defined as follows. 
\begin{defi}\emph{
We define the generating series $\DT(\sigma)$
and $\DT_{0}(\sigma)$ as 
\begin{align*}
\DT(\sigma)&\cneq
\sum_{n, \beta}\DT_{n, \beta}(\sigma)x^n y^{\beta}\in 
\kakkoS, \\
\DT_{0}(\sigma)&\cneq 
\sum_{(n, \beta)\in \Gamma_{0}}
\DT_{n, \beta}(\sigma)x^n y^{\beta}\in 
\kakkoT.
\end{align*}}
\end{defi}
Then exactly the same 
proof of Theorem~\ref{main:DT} 
shows the following. 
\begin{thm}
Under the same situation of Theorem~\ref{main:DT},
we have the 
following equalities of the generating series, 
\begin{align}\label{thm:gen1}
\DT(\sigma_{-})&=
\DT(\sigma_{+})\cdot\prod_{-(n, \beta)\in W}
\exp((-1)^{n-1}nN_{n, \beta}x^n y^{\beta}), \\
\label{thm:gen2}
\DT_{0}(\sigma_{-})&=
\DT_{0}(\sigma_{+})\cdot\prod_{-(n, \beta)\in W}
\exp((-1)^{n-1}nN_{n, \beta}x^n y^{\beta}). 
\end{align}
\end{thm}
We also have the following
corollaries. 
\begin{cor}\label{cor:gengen}
Under the same situation of Corollary~\ref{path}, we have the 
equalities of the generating series, 
\begin{align*}
\DT(\tau)&=
\DT(\sigma)\cdot \prod_{\begin{subarray}{c}
-(n, \beta) \in W_c, \\
c \in (0, 1).
\end{subarray}}
\exp((-1)^{n-1}nN_{n, \beta}x^n y^{\beta})^{\epsilon(c)}, \\
\DT_{0}(\tau)&=
\DT_{0}(\sigma)\cdot \prod_{\begin{subarray}{c}
-(n, \beta) \in W_c, \\
c \in (0, 1).
\end{subarray}}
\exp((-1)^{n-1}nN_{n, \beta}x^n y^{\beta})^{\epsilon(c)}.
\end{align*}
In particular, the quotient series
\begin{align}\notag
\DT'(\sigma)=
\frac{\DT(\sigma)}{\DT_0(\sigma)}
\in \kakkoS,
\end{align}
does not depend on general 
$\sigma\in \vV$. 
\end{cor}
Now Conjecture~\ref{conj:DTPT} is proved 
in a similar way to Theorem~\ref{main:dtpt}. 
\begin{thm}
Conjecture~\ref{conj:DTPT} is true. 
\end{thm}
\begin{proof}
We use the same notation in the proof of 
Theorem~\ref{main:dtpt}. 
By the same proof of Theorem~\ref{main:dtpt} and 
using Corollary~\ref{cor:gengen},
we have 
$$\DT'(X)= \DT'(\sigma_{\xi})=\DT'(\sigma_{\xi'})=\PT(X), $$
as expected. 
\end{proof}

\begin{rmk}
As in Remark~\ref{count0}, we have 
\begin{align*}
\DT_0(X)&=\prod_{n>0}
\exp((-1)^{n-1}n N_{n, 0}x^n) \\
&=M(-x)^{\chi(X)}.
\end{align*}
Hence we have 
$$N_{n, 0}=-\sum_{r|n}\frac{\chi(X)}{r^2}.$$
\end{rmk}

Yukinobu Toda

Institute for the Physics and 
Mathematics of the Universe (IPMU), 

University of Tokyo, 
Kashiwano-ha 5-1-5, Kashiwa City, Chiba 277-8582, Japan

\textit{E-mail address}:toda-914@pj9.so-net.ne.jp

\end{document}